\documentclass[a4paper]{article}

\usepackage{amsmath}
\usepackage{amsfonts}
\usepackage{amsthm}
\usepackage{amssymb}
\usepackage{url}
\usepackage[usenames,dvips]{color}
\usepackage{stmaryrd}
\usepackage[dvips]{graphicx}
\usepackage{psfrag, float}
\usepackage[colorlinks,pdfmark,dvips]{hyperref}
\usepackage{comment}

\usepackage{graphicx}

\newtheorem{theorem}{Theorem}[section]

\newtheorem{corollary}[theorem]{Corollary}

\newtheorem{definition}[theorem]{Definition}

\newtheorem{lemma}[theorem]{Lemma}

\newtheorem{proposition}[theorem]{Proposition}
\newtheorem{remark}[theorem]{Remark}

\def\N{\mathbb{N}}

\def\RR{\mathbb{R}}
\def\C{\mathbb{C}}
\def\exp{\text{exp}}
\def\ee{\text{e}}
\def\re{\text{Re}}
\def\im{\text{Im}}

\newcommand{\D}{\mathbb{D}}
\newcommand{\ic}{\text{i}}
\newcommand{\spec}{\text{spec}\,}

\def\e{\varepsilon}
\def\g{\gamma}
\def\r{\rho}

\def\u{{\rm u}}
\def\s{{\rm s}}

\def\p{\mu}
\def\f{\psi}
\def\fa{\phi_0}
\def\varp{\eta_1}
\def\vars{\eta_2}
\def\varpin{\hat{\eta}_1}
\def\varsin{\hat{\eta}_2}

\def\difin{\Theta}
\def\fin{\varphi}
\def\spli{\zeta}
\def\varpine{\spli_1}
\def\varsine{\spli_2}

\def\gen{\mathcal{O}}
\def\lop{\mathcal{L}}
\def\lopin{\mathcal{L}}
\def\inlop{\mathcal{L}^{-1}}
\def\inlopin{\mathcal{L}^{-1}}
\def\Ein{\mathcal{X}}
\def\EInt{\mathcal{Y}}
\def\OL{A}
\def\IL{M}
\def\UR{E}
\def\R0{\epsilon_0}
\def\bb{b}
\def\rF{\varrho}
\def\rpf{\varrho_0}
\def\iff{h}
\def\eH{\ell}

\def\F{F}
\def\ff{f}
\def\G{G}
\def\gg{g}
\def\fB{\varphi}

\title{The inner equation for generalized standard maps}
\author{I. Baldom\'a, P. Mart\'{\i}n \footnote{Universitat Polit\`ecnica de Catalunya
({\tt immaculada.baldoma@upc.edu},{\tt martin@ma4.upc.edu}).}}

\begin{document}
\maketitle

\begin{abstract}
We study particular solutions of the \emph{inner equation}
associated to the splitting of separatrices on \emph{generalized
standard maps}. An exponentially small complete expression for their
difference is obtained. We also provide numerical evidence that the
%\emph{inner equation}
inner equation
provides quantitative information of the
splitting of separatrices even in the case when the limit flow does
not.
\end{abstract}

%\begin{keywords}Inner equation, exponentially small phenomena, splitting of
%separatrices, asymptotic formula, numerical
%computations.\end{keywords}

%\begin{AMS}37J10, 34C37, 37C29, 34E10, 34M37\end{AMS}

\bibliographystyle{alpha}

\pagestyle{myheadings}
\thispagestyle{plain}
\markboth{I. Baldom\'a, P. Mart\'{\i}n}{Inner equation for area preserving maps}

\section{Introduction}
The phenomenon of the splitting of separatrices occurs when a
dynamical system having an invariant object (a fixed point, a
periodic orbit, a torus, etc.) with coincident branches of its
stable and unstable invariant manifolds (a separatrix), is
perturbed. Generically, a new invariant object of the perturbed
system arises which still possesses stable and unstable invariant
manifolds but which no longer coincide.

The problem of measuring the size of this splitting is long standing
in Dynamics. It is related to the existence of transversal
homoclinic points and, consequently, with the non-integrability and
with the size of the stochastic zone of the system under study.

The most popular tool for measuring the splitting of separatrices is
the Melnikov approach~\cite{Melnikov63}. It is based on classical
perturbation theory and provides a first order approximation for the
splitting by using the distance between the stable and unstable
invariant manifolds of the perturbed system. Nevertheless there are
plenty of interesting (and in some sense generic) situations where
this approach fails: when the Melnikov function does not predict
correctly the size of the splitting or when no Melnikov function is
available, for instance, when integrable systems near simple
resonances are perturbed. In this case, Poincar\'{e} already
detected in~\cite{Poincare99} that the separatrix splits but it
turns out that the size of this splitting is exponentially small in
the perturbation parameter, what it is usually known as a
beyond-all-orders phenomenon. Consequently a direct application of a
first order perturbation theory never will be able to provide a good
estimation for this exponentially small splitting. There are other
settings, related for instance to Arnold diffusion and fluid
transport, when the splitting of separatrices is exponentially small
in the perturbation parameter, but from now on we will restrict
ourselves to the case of near identity, analytic, area-preserving
maps.

\subsection{Exponentially small splitting of separatrices in analytic maps}
Throughout this introduction we will avoid precise statements and
technicalities but we will give the main ideas about the
exponentially small phenomenon.

Consider an area preserving analytic map, close to the identity, that is, a map which can be written as
\begin{equation}\label{APM}
G(z,h)=z+h g(z,h),\;\;\;\;z\in \RR^2
\end{equation}
where $h$ is a small parameter and $g(0,h)=0$, so that the origin is
a fixed point for any value of $h$. Assume also that the origin is a
weakly hyperbolic fixed point. Namely, redefining the parameter $h$
if necessary, the eigenvalues $\lambda, \lambda^{-1}$ of $DG(0)$ are
of the form $\lambda = \text{e}^{h}=1+O(h)$. In this case, there
exist $W^{\s}$ and $W^{\u}$, the stable and unstable invariant
manifolds of the origin, respectively. The goal is to measure the
discrepancy between these invariant manifolds. Notice that, since
for $h=0$ the map $G(z,0)=z$, this is a beyond-all-orders
phenomenon. The strategy is to not consider the first approximation
of the map $G$ as simply taking $h=0$, but as the time $h$ map of
the vector field
\begin{equation}\label{APMVF}
z'=g(z,0).
\end{equation}
It can be seen, for instance~\cite{FS90}, that this approximation
holds under generic and checkable assumptions. If the vector
field~\eqref{APMVF} possesses a homoclinic connection $\gamma_0$
associated to the origin (the fixed point), then one expects that
the exponentially small splitting of separatrices phenomenon arises
for maps of the form~\eqref{APM}. In fact in~\cite{FS90} is proved
that, for any $p\in W^{\s}$
\begin{equation}\label{boundexp}
\text{dist}(p, W^{\u}) \leq K_{\sigma} \ee^{-2\pi \sigma/h}
\end{equation}
being $\sigma>0$ and $K_{\sigma}$ a constant depending on $\sigma$
and $p$, but independent of $h$. Nevertheless this upper bound is
not useful for deciding whether the separatrix $\gamma_0$ splits or
not. It turns out to be mandatory to obtain an expression for the
asymptotic behavior of the splitting.

We emphasize here that, even when the distance between $W^{\s}$ and
$W^{\u}$ seems a good choice for measuring the splitting, it depends
on the point $p$. This is because this measure does not exploit the
area preserving character of our map. There are several quantities
more appropriate for this task. One of them is the Lazutkin
invariant (see formula~\eqref{def:Lazutkininvariant} in
Section~\ref{firstasumptions}) which is related to the angle between
both $W^{\s}$ and $W^{\u}$ at a homoclinic point. An upper bound
similar to~\eqref{boundexp} for the Lazutkin invariant can be
obtained but with $K_{\sigma}$ depending only on $\sigma$.

If the asymptotic behavior for  the splitting has to be proved, the
first question that arises from~\eqref{boundexp}, is how much bigger
$\sigma$ could be. To find this optimal value of $\sigma$ one has to
know the analyticity domain of $\gamma_0$, the homoclinic connection
of the vector field~\eqref{APMVF}. It is proven in~\cite{Fontich95}
that $\gamma_0$ has complex singularities, henceforth it is analytic
in a maximal complex strip $\{ t \in \C : |\im t| < \sigma_0\}$. The
bound~\eqref{boundexp} holds for any $\sigma<\sigma_0$, changing
$K_{\sigma}$ appropriately. Notice that if we take any fixed
$\sigma_* < \sigma_0$, we do not obtain a sharp upper bound, simply
because the result also holds for $\sigma$, with $\sigma_*<
\sigma<\sigma_0$ and henceforth, taking $h$ small enough we get a
better estimate than the previous one. In consequence any asymptotic
formula will require taking $\sigma$ arbitrarily close to $\sigma_0$
as a function of $h$.

The key point for proving the bound~\eqref{boundexp} is to obtain
good parameterizations for the invariant manifolds $W^{\s}$,
$W^{\u}$, which are analytic in the complex strip $\{ t \in \C :
|\im t| < \sigma\}$ with $\sigma<\sigma_0$. The natural
parameterization for the invariant manifolds are functions
$\gamma^{\u,\s}(t)$ that satisfy
\begin{equation}\label{INVCOND}
G(\gamma^{\u,\s}(t),h)=\gamma^{\u,\s}(t+h).
\end{equation}
Notice that the homoclinic connection $\gamma_0$ satisfies this
invariance equation for the time $h$ flow of the vector
field~\eqref{APMVF}. As we have mentioned in the above paragraph, to
obtain an asymptotic formula for the splitting it is necessary to
find solutions of the invariance equation~\eqref{INVCOND} defined
for values of $t$ arbitrarily close to $\sigma_0$ as a function of
$h$. Since the strip is limited by the singularities of the
homoclinic connection $\gamma_0$, this study becomes harder when the
values of $t$ are closer to these singularities. The \emph{inner
equation} is a suitable approximation of the invariance
equation~\eqref{INVCOND} for values of $t$ close to these
singularities.

The main goal of this paper is to derive the  inner equation for a
large set of area-preserving maps (the so called \emph{generalized
standard maps}) and to obtain information about some special
solutions and their difference. This is a first step in the proof of
an asymptotic formula for the splitting of the invariant manifolds
for these maps, but obtaining this formula is beyond the scope of
this work. Nevertheless we will provide some numerical results
which, combined with heuristic arguments (see
Section~\ref{numericalresults}, in particular
formula~\eqref{innerinvLazutkin1}) support the relation between the
splitting and the inner equation.

\subsection{The inner equation. An overview}\label{introduction:innerequation}
The study of the inner equation has been at the heart of the proof
of the exponentially small splitting of separatrices in many
examples, for maps~\cite{Gel99,MSS11a,MSS11b} as well as for
flows~\cite{GOS10}.

In the case of area-preserving analytic maps, the use of the inner
equation dates back to~\cite{Laz84}, where a scheme to obtain an
asymptotic formula for the splitting of separatrices of the Chirikov
standard map was established. In that paper, a particular instance
of the inner equation was introduced: the so called
\emph{semi-standard map}. Further development of the ideas
in~\cite{Laz84} lead to the first rigorous proof of the asymptotic
formula for the Chirikov standard map in~\cite{Gel99}. A brief
discussion on the splitting size on the Chirikov standard map can be
found in~\cite{Gel00a}. From the same authors, the survey on
exponentially small phenomena~\cite{GL01} introduces, among other
things, the inner equations associated to \emph{polynomial standard
maps}, and lists in an informal way asymptotic formulas for the
splitting of separatrices in those cases. It is also remarkable that
in the paper~\cite{GS01} resurgence theory is applied to the study
of the solutions of the inner equation associated to the area
preserving H\'enon map. This paper is strongly related
to~\cite{Gel00b}. Also in the study of perturbation of the McMillan
map~\cite{MSS11a,MSS11b} resurgence methods were applied to study
the inner equation. Summarizing, one can find rigorous results on
the inner equation in~\cite{Gel99, GS01,MSS11b} in particular
examples which are covered under our present work, which also
includes and generalizes the ones present in~\cite{GL01} and the
numerical study~\cite{GS08}.

In the case of flows, the inner equation has been a successful tool
to measure the splitting of separatrices when the Melnikov function
fails to predict the size of the splitting, like in the rapidly
forced pendulum (see~\cite{Gel00c,GOS10} or~\cite{BFGS11} for a
generalization to arbitrary polynomial Hamiltonian systems of one
and a half degrees of freedom, following the study on the inner
equation in~\cite{Bal06}). A different technique based on continuous
averaging to study the exponentially small behavior of the splitting
can be found in~\cite{Treschev}.

The purpose of the present paper is twofold, a combination of
rigorous theoretical results in a general setting and numerical
experiments avoiding lengthy proofs in particular examples. One of
the numerical examples shows a a type of behavior that is not
covered by the surveys~\cite{GL01,GS08} (see the end of this
section).

We study some second order difference equations, called inner
equations, which have the form either
\begin{equation*}
\phi(z+1)-2\phi(z)+\phi(z-1)=-\phi^n(z) + \G(\phi(z))
\end{equation*}
or
\begin{equation*}
\phi(z+1)-2\phi(z)+\phi(z-1)=-\ee^{n \phi(z)} + \G(\ee^{\phi(z)}),
\end{equation*}
depending on the class of maps under consideration, and where
$\G(w)$ is an analytic function such that $\G(w) = \gen(w^{n+1})$.

These equations appear, in particular, in the problem of
exponentially small splitting of separatrices in \emph{generalized
standard maps} (see next Section for definitions), but they can
appear in studies of other types of maps (with parabolic fixed
points, for instance), and, with this applicability in mind, we
consider them in their full generality (see
equations~\eqref{innereq} and~\eqref{innereqt}). In particular, our
present results generalize those on the inner equations appearing
in~\cite{Laz84,Gel99,Gel00a,GL01,GS01,MSS11a,MSS11b}. It is
important to remark that in the previous literature on the subject
the symmetries of the particular problems under consideration where
exploited extensively in the proofs. Our present formulation does
not rely on additional symmetries, making it suitable for
applications. In particular, we provide all the technical details
and complete proofs of the statements concerning the inner equations
and their solutions. As a side comment for the specialists, there
are several technical improvements in the proofs of our theoretical
results, which we expect can be applied in other problems related to
difference equations.

We describe a large set of formal solutions of these inner
equations, from which some \emph{true} solutions are obtained, and
we derive a complete formula for their difference. The main results
are collected in Section~\ref{sec:mainresults}, while
Section~\ref{sec:generalizedstandardmaps} plays the role of a more
detailed introduction of the problem and description of some of the
known results. Sections~\ref{sec:formalsolution}, \ref{sec:solinner}
and~\ref{sec:difference} are devoted to prove the theoretical
results, while Section~\ref{numericalresults} contains the numerical
results with a non-rigorous exposition of their relation to the
developed theory. It should be remarked that the relation between
the inner equation and the actual computation of the splitting, in
the particular cases where proofs are available
(see~\cite{Gel99,MSS11a,MSS11b}), is lengthy and full of
technicalities. Our exposition here tries to give the reader an idea
of the link between the inner equation and splitting size, by making
very strong assumptions, in order to explain the obtained numerical
results. These assumptions are fully proved in the literature for
the Chirikov standard map and the McMillan map.

The numerical experiments have been conducted to test the
applicability of the theoretical results. Although academic in
nature, they show the relation between the splitting of separatrices
and the difference between two solutions of the inner equation.
Moreover, the main example exhibits a behavior that is not covered
by the surveys~\cite{GL01,GS08}. In this example, given by the map
\[
\begin{pmatrix} x \\ y \end{pmatrix} \mapsto
\begin{pmatrix} x + y + \e (x-x^3) - \e^2 x^7\\ y +\e (x-x^3) - \e^2
x^7\end{pmatrix},
\]
where $\e$ is a small parameter, although the size of the splitting
is much larger than the guess suggested by~\cite{FS90}, the leading
term of its asymptotic behavior is provided by the
%\emph{inner equation}.
inner equation. As a matter of fact, the splitting size in this
example behaves asymptotically when $\e \to 0$ as
$$
\frac{A}{h^{10/3}}\exp\left(-\frac{\pi^2}{h}+\frac{2^{5/4}\sqrt{\pi}\Gamma(3/4)^2)}{h^{1/2}}\right)
(1+ \text{higher order terms}),$$ where $\e = 4 \sinh^2(h/2)$ and
$A$ is a constant related to some inner equation, while the na\"{\i}ve guess provided by the limit
flow (see Section~\ref{sec:generalizedstandardmaps} for details), in
this case the Duffing equation $\ddot x = x - x^3$, would be
exponential with exponent $-\pi^2/h$. That is, the correction term
is larger than any power of~$h$.
See Sections~\ref{discrepantexample} and~\ref{numericalresults}.

We remark that although the computation of the actual splitting has
been performed by using the multiple precision package PARI-GP, the
computation of the leading term has been achieved by using the
standard {\tt long double} precision in~{\tt C}.

\section{Generalized standard maps and exponentially small splitting of separatrices}
\label{sec:generalizedstandardmaps}
\subsection{Generalized standard maps}
\label{firstasumptions} We will say that an area preserving map
$(x^*,y^*) = \F(x,y)$ is a \emph{generalized standard map} if it can
be written in the form
\begin{equation}
\label{def:generilizedstandardmaps} \left\{
\begin{aligned}
x^*&= x+y+ \ff(x,h), \\
y^*&=y+\ff(x,h),
\end{aligned}
\right.
\end{equation}
where $h$ is a small parameter. We will assume that $f$ depends
analytically in its arguments in $|h|< h_0$, $|x| < \rho_0$, for
some fixed $h_0, \rho_0 >0$. We will be interested in the case when
the origin is a fixed point of $\F$, that is, $\ff(0,h) = 0$. Moreover,
we will assume the origin to be weakly hyperbolic, although our
study may be applied also to the case of a parabolic fixed point.

The parameter~$h$ is chosen in such a way that $\spec D\F(0,0) =\{
\ee^{h}, \ee^{-h}\}$. This last condition is equivalent to impose
$\ff'(0,h) = \frac{\partial}{\partial x} \ff (0,h) = \e$, with $\e = 4
\sinh^2(h/2)$. We further assume that
\begin{equation}
\label{hypothesisonf} \ff(x,h) = \sum_{k\ge 0} \ff_k(x) h^{k+2} = \e
\ff_0(x) + O(h^3 x).
\end{equation}
Under these conditions, the map~\eqref{def:generilizedstandardmaps}
can be written as a close to the identity map: with the scaling
$\tilde x = x$, $h \tilde y = y$, it becomes (using again $x$ and
$y$ as variables)
\begin{equation}
\label{def:generilizedstandardmapsclosetotheidentity} \left\{
\begin{aligned}
x^*&= x+ h y+  O(h^2 x), \\
y^*&=y+h \ff_0(x) + O(h^2x).
\end{aligned}
\right.
\end{equation}
When $h$ is small, the
map~\eqref{def:generilizedstandardmapsclosetotheidentity} is well
approximated by the time~$h$ map of the flow of the Hamiltonian
system
\begin{equation}
\label{def:limitfloweq}
 \left\{
\begin{aligned}
\dot x &= y, \\
\dot y &= \ff_0(x).
\end{aligned}
\right.
\end{equation}
We assume that the origin in~\eqref{def:limitfloweq}, which is a
fixed point, possesses a homoclinic connection, $\gamma_0(t) =
(x_0(t),y_0(t))$. By a shift on $t$, we can choose $\gamma_0$ such
that $x_0$ is an even function, that is, $\gamma_0$ intersects
transversally the line $\{y = 0\}$ at $t=0$. The invariant manifolds
of the origin for the
map~\eqref{def:generilizedstandardmapsclosetotheidentity} are close
to this homoclinic connection. Hence, if $h$ is small, by the
conservation of the area, they must intersect. It is not difficult
to check that the expansions in powers of $h$ of the stable and
unstable curves coincide. As a consequence, the expansion of the
angle of intersection in powers of $h$ vanishes, which, in view of
the analytic nature of the problem, suggests that this angle may
have an exponentially small behavior in $h$. In fact, Fontich and
Sim\'o, in~\cite{FS90}, obtained an exponentially small upper bound
for the angle. They showed that if $\gamma_0$ is analytic in the
complex strip $\{|\im t| < \sigma_0\}$ and the map $\F$ is defined
around the homoclinic orbit, then, for any $0 < \sigma < \sigma_0$,
the distance between the stable and the unstable manifold of the
origin of~\eqref{def:generilizedstandardmapsclosetotheidentity} is
bounded by $K_{\sigma} \ee^{-2\pi \sigma/h}$, for any $0 < h <
h_{\sigma}$, where $K_{\sigma}$ and $h_{\sigma}$ are positive
constants depending on $\sigma$ and $K_{\sigma}$ depends also on the point
where this distance is measured. Restoring to the original
variables, the same applies to the invariant manifolds of the origin
of~\eqref{def:generilizedstandardmaps}.

Equivalently, a \emph{natural parametrization} $\gamma(t) =
(x(t),y(t))$ of the invariant manifolds of the origin
of~\eqref{def:generilizedstandardmaps}, when
condition~\eqref{hypothesisonf} is satisfied, that is, a
parametrization satisfying $\F \circ \gamma (t) = \gamma(t+h)$, must
be a solution of the difference equation
\begin{equation}
\label{eq:invariance} x(t+h) -2 x(t) +x(t-h) = \ff(x(t),h),
\end{equation}
with $y(t) = x(t)-x(t-h)$. This equation implies that the curve
$\gamma = (x,y)$ is invariant by $\F$ and the action of $\F$ on
$\gamma$ is conjugated to the shift on the parameter $t$: $t \mapsto
t+h$ . One must supply additional conditions on $\gamma$ to obtain
the invariant stable and unstable curves: if $\gamma$ is the
unstable manifold (resp. stable) of the origin then $\lim_{t \to
-\infty} x(t) = 0$ is required (resp. $\lim_{t \to \infty} x(t) =
0$).

Since the left hand side of the invariance
equation~\eqref{eq:invariance} is formally
$$
x(t+h) -2 x(t) +x(t-h) = 4 \sinh^2\left(\frac{h}{2}
\frac{\partial}{\partial t}\right) (x)(t) = h^2 \ddot x (t) +
O(h^4),
$$
it can be approximated, when $h$ is small, by the second order
differential equation
\begin{equation}
\label{def:limitfloweq2} \ddot{x} = \ff_0(x),
\end{equation}
which is nothing more than~\eqref{def:limitfloweq}.

In order to measure the difference between the invariant manifolds,
it is often used the \emph{Lazutkin invariant} at a homoclinic point
$p = \gamma^u(0) = \gamma^s(0)$,
\begin{equation}
\label{def:Lazutkininvariant} \omega(p) = \det \left(\frac{d}{dt}
\gamma^u (0), \frac{d}{dt} \gamma^s(0)\right),
\end{equation}
being $\gamma^{u,s}(t)$ natural parametrizations of the unstable and
stable manifolds. Unlike the angle between the invariant curves,
$\omega(p)$ is a symplectic invariant and only depends on the
homoclinic orbit, not on the specific point $p$. Another symplectic
invariant quantity that can be used to measure the splitting of the
separatrices is the area of the lobe between two consecutive
homoclinic points.

Since an upper bound of the splitting of the separatrices is known,
the question of its asymptotic behavior when~$h$ tends to~$0$
arises. Some well known examples in the literature where this
formula is available are briefly summarized in the next subsection.

\subsection{Examples of generalized standard maps with
exponentially small splitting of separatrices}
\label{standard-mcmillan}

There are not many examples with a complete proof of an asymptotic
formula for the splitting of separatrices in area preserving maps.
Here we quote two. There is a more abundant literature about
splitting of separatrices in Hamiltonian systems with one and a half
degrees of freedom
(see~\cite{HolMS91,DelsSeara97,DelsGelJorSea,Treschev,LocMS03,OliveSS,DelG04})

The first example is the Chirikov standard map, introduced by
Chirikov as a basic model of the motion of a system close to a
nonlinear resonance (see, for instance, \cite{Chi79}). It
corresponds to take $\ff(x,h) = \e\sin(x)$, with $\e = 4
\sinh^2(h/2)$. This map is in fact defined in the annulus, and the
limit flow~\eqref{def:limitfloweq} is a pendulum with the saddle at
the origin. The separatrix of the pendulum is analytic in the strip
$\{|\im t| < \pi/2\}$ and has a singularity at $t= \ic \pi/2$. The
symmetries of the problem imply that there is a homoclinic point $p$
on the line $x= \pi$.

In~\cite{Gel99}, Gelfreich proved, following the scheme developed by
Lazutkin in~\cite{Laz84}, that
\[
\omega(p) \asymp \frac{4\pi}{h^2} \ee^{-\pi^2/h} \sum_{k\ge 0}
h^{2k} \omega_k,
\]
where the series in the right hand side is asymptotic. In
particular, the exponent in the exponential is well predicted by
Fontich-Sim\'o theorem in~\cite{FS90}.

The second example is the perturbed McMillan map. The McMillan map
itself was introduced in~\cite{McM71} in the modelization of
particle accelerator dynamics. In~\cite{DR98,MSS11a,MSS11b},
perturbations of the McMillan family of the form
\begin{equation}
\label{def:mcmillan} \left\{
\begin{aligned}
x^*&= y, \\
y^*&= -x + \frac{2\cosh(h) y}{1+y^2} +\tilde \e V'(y),
\end{aligned}
\right.
\end{equation}
are considered, with $V(y) = \sum_{k\ge 2} V_k y^{2k}$ analytic in a
neighborhood of $y=0$. In the above formula, $h$ is the Lyapunov
exponent of the origin, which is the small parameter, and $\tilde
\e$ is independent of $h$ and not necessarily small. The McMillan
map is obtained when $\tilde \e = 0$ and is integrable with a
polynomial first integral. See~\cite{DR98} for more details about
the McMillan map.

With a linear change of coordinates, the map \eqref{def:mcmillan}
can be written in the form~\eqref{def:generilizedstandardmaps} with
\begin{align}\label{fformcmillan}
 \ff(x,h) &= \e \frac{x-2 x^3}{1+ \e x^2} +
\frac{\tilde\e}{\e^{1/2}} V'\left( \e^{1/2} x \right) \notag\\
&= \e(x-(2-4\tilde \e V_2) x^3) -\e^2(x^2-(2+6\tilde \e
V_3)x^5)+O(\e^4),
\end{align}
where, again, $\e = 4 \sinh^2(h/2)$. The limit
flow~\eqref{def:limitfloweq2} is the Duffing equation
$$
\ddot{x} = x-(2-4\tilde \e V_2) x^3,
$$
with homoclinic $x_0(t) = \alpha/\cosh(t)$, $\alpha = (1-2\tilde \e
V_2)^{-1/2}$ (assuming $|\tilde \e| < (2 V_2)^{-1}$). Its closest to
the real line singularities are located at $\pm \ic \pi/2$.
In~\cite{MSS11a,MSS11b}, improving a partial result in~\cite{DR98},
it was proven that, if $\hat V(2\pi) \neq 0$, where
\[
\hat V (\zeta)= \sum _{k\ge 2} V_k \frac{\zeta^{2k-1}}{(2k-1)!}
\]
is the Borel transform of $V$, then the invariant manifolds to the
origin of~\eqref{def:mcmillan} split when $\tilde \e \neq 0$ and the
Lazutkin invariant of a particular homoclinic orbit satisfies
\[
\omega \asymp \frac{4\pi\tilde \e}{\beta^2 h^2} e^{-\pi^2/h}
\sum_{k\ge 0} h^{2k} B_k^+(\tilde \e),
\]
where the functions $B_k^+$ are analytic around $\tilde \e = 0$,
$\beta^2 = 1 - 2\tilde \e V_2/\cosh h$ and $B_{0}^+(\tilde \e) =
4\pi^{2}\hat V(2\pi) + O(\tilde \e)$. If the map is written in the
form~\eqref{def:generilizedstandardmaps}, with the function $f$
given in~\eqref{fformcmillan}, the Lazutkin invariant has an
additional $h^2$ in the denominator. Again, the exponent of the
exponential is well predicted by Fontich-Sim\'o theorem.

\subsection{Numerical studies for polynomial generalized standard
maps} \label{numericalGS}

In~\cite{GS08}, Gelfreich and Sim\'o presented a detailed numerical
study of the splitting of the separatrices of the generalized
standard map~\eqref{def:generilizedstandardmaps} in the case $\ff(x,h)
= \e p(x)$, with $p (x) = \sum_{k=1}^n p_k x^k$ a polynomial of
degree $n$ with $p_1=1$ (which implies $\ff'(0,h)= \e$) and $p_n < 0$.
Is is also assumed that there is a homoclinic curve to the origin in
the limit flow system~\eqref{def:limitfloweq2}.

Then, via numerical experiments, the authors showed that the
asymptotic behavior of the Lazutkin invariant depends only on the
relative position of the singularities of the homoclinic solution
of~\eqref{def:limitfloweq2}, on the degree $n$ of the polynomial $p$
and on the coefficient $p_n$:
$$
\omega \asymp \frac{C_n}{|p_n|^{\nu/2} h^{\nu}} \ee^{-2\pi \rho/h}
\tilde \omega(h)+ \dots,
$$
where $\nu = 2(n+1)/(n-1)$, $\rho$ is the minimum distance to the
real line of the singularities of the homoclinic
of~\eqref{def:limitfloweq2}, $\tilde \omega(h) \not \equiv 0$ is
either a constant, a periodic function or a quasiperiodic function
of $1/h$, depending only on the number of singularities at~$|\im t|
= \rho$ and their relative positions and $C_n$ depends only on $n$.

Also in this case, the exponential behavior is well predicted by
Fontich-Sim\'{o}'s theorem.

\subsection{A discrepant example. Numerical observations}
\label{discrepantexample}

We introduce the generalized standard
map~\eqref{def:generilizedstandardmaps} induced by
\begin{equation}
\label{gsm7} \ff(x,h) = \e(x-x^3)-\e^2 x^7.
\end{equation}
Note that this map possesses terms in $\e^2$, like the McMillan map
has (see~\eqref{fformcmillan}). Unlike the McMillan case, the
function defining this map is entire.

The limit flow~\eqref{def:limitfloweq2} for this map is also a
Duffing equation, in this case $\ddot x = x-x^3$, with homoclinic
$x_0(t) = \sqrt{2}/\cosh(t)$, whose singularities are located at the
same place of the homoclinic of the McMillan map, being $\pi/2$
their minimum distance to the real line. Hence, one could be tempted
to infer that the exponential behavior of the Lazutkin invariant is
of order $e^{-\pi^2/h}$.

However, our numerical experiments suggest that the Lazutkin
invariant at the \emph{first} homoclinic point over the line $y=0$,
in the topology of the unstable manifold, behaves like
\begin{equation}
\label{omegamap7} \omega \asymp \frac{A}{h^{10/3}} \ee^{-2\pi
\rho(h)/h} + \dots,
\end{equation}
where
\begin{equation}
\rho(h) = \frac{\pi}{2}-\frac{2^{1/4}\Gamma(3/4)^2}{\sqrt{\pi}}
h^{1/2}+ O(h^{3/2})
\end{equation}
and $A=871.683\dots$. In particular, the size of the
Lazutkin invariant is much larger than the na\"{\i}ve guess, which,
in turn, suggests that the approximation of the invariant manifolds
provided by the limit flow~\eqref{def:limitfloweq2} is not good
enough to predict the asymptotic formula of the splitting.
Section~\ref{numericalresults} is devoted to explain these numerical
experiments. In particular, we will conjecture the source of the
function~$\rho(h)$ and the origin and computation of the
constant~$A$.

\subsection{Inner equation for generalized standard maps}
\label{subsec:innerequation}

In all the aforementioned examples, the constants $\omega_0$,
$B_0^+(\tilde \e)$, $C_n$ and $A$ in the leading term of the
asymptotic behavior of the Lazutkin invariant are related to a
suitable
%\emph{inner equation},
inner equation,
whose solutions provide better approximations of the invariant
manifolds for values of $t$ in some regions of $\C$ than the one
provided by the limit flow~\eqref{def:limitfloweq2}. Even in the
case of the generalized standard map defined by~\eqref{gsm7}, where
the limit flow~\eqref{def:limitfloweq2} does not provide enough
information, the numerically evaluated constant~$A$
in~\eqref{omegamap7} is obtained from such an
%\emph{inner equation}.
%
inner equation.

In order to be able to construct the
%\emph{inner equation}
inner equation
we will impose several conditions to the function
defining the generalized standard map.

Let $\F$ be a generalized standard map of the
form~\eqref{def:generilizedstandardmaps}, induced by a function
$\ff(x,h) = \sum_{k\ge 0} \ff_k(x) h^{k+2}$, satisfying the hypotheses
in Section~\ref{firstasumptions}. We furthermore assume:

\begin{itemize}
\item[\textbf{HP1}] For each $k\ge 0$,
$\ff_k(x)=\sum_{j=1}^{d_k} f_{k,j} x^j$, with $f_{k,d_k}\neq 0$.
\item[\textbf{HP2}] The function $k\mapsto (d_k-1)/(k+2)$ has a global maximum on $\N$.
Let $I\subset \N$ be the set where this maximum is achieved.
\end{itemize}

Hypothesis~\textbf{HP2} implies a restriction in the rate of growth
of the degree of each of the polynomials $\ff_k$, which can be at
most linear in $k$. We also remark that, combining hypothesis
\textbf{HP1} and \textbf{HP2} with the fact that $\ff$ is analytic in
the bidisk $\D_{\rho_0} \times \D_{h_0}$, one obtains that the
domain of analyticity with respect to $x$ depends on $h$ and tends
to be the whole complex plane when $h$ tends to $0$.

We fix $\chi \in \C$. We introduce the new unknown
$\phi(z)$ defined by $x(\chi+hz) = h^{-\alpha} \lambda \phi(z)$, with
$$
\alpha = \frac{k+2}{d_{k}-1}, \;\;\;\text{for any } k \in I
$$
and $\lambda$ a parameter to be determined later.
Note that, by definition of $I$ in \textbf{HP2}, $\alpha$ is indeed independent of $k\in I$. The invariance
equation~\eqref{eq:invariance} becomes
\begin{equation}
\label{inveqinnervariables} \phi(z+1) -2 \phi(z) +\phi(z-1) =
h^{\alpha} \lambda^{-1} \ff(h^{-\alpha} \lambda \phi(z),h).
\end{equation}
With the standing hypotheses, the right hand side above admits an expansion of suitable positive powers of $h$ as follows:
\begin{align*}
h^{\alpha} \lambda^{-1} \ff(h^{-\alpha} \lambda \phi(z),h) &= \sum_{k\ge 0} \sum_{j=1}^{d_k}\ff_{k,j} h^{-\alpha (j-1)} \lambda^{j-1} \phi^j(z) h^{k+2} \\
&= \sum_{k\ge 0} h^{k+2-\alpha(d_k-1)} \left (f_{k,d_k} \lambda^{d_k-1} \phi^{d_k}(z) +  \sum_{j=1}^{d_k-1} h^{\alpha(d_k-j)} f_{k,j} \lambda^{j-1} \phi^{j}(z)\right)
\\ &= \sum_{k\in I} f_{k,d_k} \lambda^{d_k-1} \phi^{d_k} + \gen(h^{\min\{1,\alpha\}}),
\end{align*}
where in the last equality we have used the definition of $\alpha$ and $I$.
The inner equation is obtained by keeping only the first term in $h$ in the right side of \eqref{inveqinnervariables}:
\begin{equation}\label{preinner}
\phi(z+1) -2 \phi(z) +\phi(z-1) = \sum_{k\in I} f_{k,d_k} \lambda^{d_k-1} \phi^{d_k}.
\end{equation}
Let $n=\min\{d_k : k\in I\}$. To simplify the notation we introduce the coefficients $\tilde{\G}_k$ such that
\begin{equation*}
\sum_{k\in I} f_{k,d_k} \lambda^{d_k-1} \phi^{d_k} = \sum_{k\geq n} \tilde{\G}_k \lambda^{k-1} \phi^{k}.
\end{equation*}
Now we take $\lambda$ such that
$
\lambda^{n-1} = -(\tilde{\G}_{n})^{-1}.
$
With this choice, the
%\emph{inner equation}
inner equation
associated to the generalized standard map
is
\begin{equation}
\label{innereqpresentation} \phi(z+1) -2 \phi(z) +\phi(z-1) =
 -\phi^{n}(z) + \sum_{k\geq n+1} \G_k \phi^k(z)
\end{equation}
being $\G_k = \tilde{\G}_k \lambda^{k-1}$. Notice that $\G(\phi) := \sum_{k\geq  n+1} \G_k \phi^k $ is analytic in a neighborhood of $\phi=0$.

In the trigonometric case one can proceed analogously. Indeed,
assume that $\ff(x,h) = \sum_{k\ge 0} \ff_k(x) h^{k+2}$, with $f$
satisfying:
\begin{itemize}
\item[\textbf{HT1}] For each $k\ge 0$, $\ff_k(x)
= \sum_{j=-d_k}^{d_k} \ff_{k,j} \ee^{\ic j x}$ is a trigonometric
polynomial of degree $d_k \ge a$, with $\ff_{k,d_{k}} \neq 0$.
\item[\textbf{HT2}] The function $k\mapsto d_k/(k+2)$ has a global maximum on $\N$.
Let $I\subset \N$ be the set where this maximum is achieved.
\end{itemize}
For any $\chi \in \C$, we define $\phi(z)$ by $x(\chi + hz)=-\ic \log
\big (h^{\alpha} \lambda\big ) + \ic \phi(z)$ with
\begin{equation}
\alpha=\frac{k+2}{d_{k}},\;\;\;\;\; \;\;\;\text{for any } k \in I
\end{equation}
and $\lambda$ a parameter. Then, the invariance equation~\eqref{eq:invariance} becomes
\begin{equation}
\label{inveqinnervariablestrig} \phi(z+1) -2 \phi(z) +\phi(z-1) =
-\ic \ff(-\ic \log \big (h^{\alpha}\lambda\big) + \ic \phi(z),h).
\end{equation}
As in \eqref{preinner}, the
%\emph{inner equation}
inner equation
is the above equation
when $h\to 0$. In this case, taking $\lambda$ appropriately, one obtains
\begin{equation}
\label{innereqpresentation-trig} \phi(z+1) -2 \phi(z) +\phi(z-1) =
 - \ee^{(n-1)\phi(z)} + \sum_{k\geq n} G_k \ee^{k\phi(z)},
\end{equation}
where $n-1=\min\{d_k : k\in I\}$. The discrepancy in the definition of $n$ in both cases
allows us to make an unified treatment of the problem in next sections.

Since the original invariance equation~\eqref{eq:invariance} is
autonomous, the
%\emph{inner equation}~\eqref{innereqpresentation}
inner equation~\eqref{innereqpresentation} and~\eqref{innereqpresentation-trig}
does not depend on the choice of the complex number~$\chi$
introduced with the new unknown $\phi$. Nevertheless, this complex
number is essential when the size of the splitting of separatrices
is studied and has to be well chosen. Roughly speaking, it will measure the exponentially smallness
of the splitting which turns out to be $\gen (h^{\nu} \ee^{-2\pi \im \chi/h}))$ for some
$\nu\in \RR$. This asymptotic behavior has only been proved for particular maps (see Section~\ref{standard-mcmillan})
but there are numerical evidences, Sections~\ref{numericalGS} and~\ref{discrepantexample},
that it also holds in a more general setting.
We plan, in a future work, to prove it for the generalized standard maps.

In the examples presented in
Sections~\ref{standard-mcmillan} and~\ref{numericalGS}, $\chi$ is
chosen to be the location of the closest to the real line
singularity of the homoclinic solution $\gamma_0$ of the limit
flow~\eqref{def:limitfloweq2}. In example in
Section~\ref{discrepantexample}, $\chi$ is also related to the
singularities of a homoclinic solution of some flow, which is not
longer~\eqref{def:limitfloweq2} but $\ddot{x}=x-x^3 -\varepsilon x^7$.

\subsubsection{Some examples of the inner equation}\label{Examplesinner}
Here we show how the inner equation is derived for some examples.

The first one is the map introduced in
Section~\ref{discrepantexample}. Its inner equation is
\begin{equation}
\label{innerdiscrepant}
\phi(z+1) -2 \phi(z) +\phi(z-1) = -\phi^7(z).
\end{equation}
Indeed, in this case $f(x,h)=\varepsilon (x-x^3) + \varepsilon^2 x^7$ with $\varepsilon = 4\sinh^2 (h/2)$ . Therefore,
$f_0(x)=x-x^3$, $f_{2k}(x) = f_{2k,1}x- f_{2k,3} x^3 - f_{2k,7}x^{7}$ and $f_{2k-1}(x)=0$, for $k\geq 1$, which implies that
$d_{0}=3$, $d_{2k}=7$ and $d_{2k-1}=0$ for $k\geq 1$. In this situation, it is clear that $n=7$, $\alpha=2/3$ and the set $I=\{2\}$, therefore
the right hand side of equation~\eqref{preinner} is $f_{2,7} \lambda^6 \phi^7$ and defining $\lambda$ adequately we encounter
equation~\eqref{innerdiscrepant}.

Now we compute the inner equation for the generalized standard map induced by $f(x,h)=\varepsilon (x-x^3)$. In this case
$f_{2k}(x) = f_{2k,1}x-f_{2k,3}x^3$, $d_{2k}=3$, $f_{2k+1}(x)=0$ and $d_{2k+1}=0$, for $k\geq 0$
and this implies that $n=3$, $\alpha=1$ and the set $I=\{0\}$. Then, the
right hand side of equation~\eqref{preinner} is $f_{0,d_{0}}\lambda^2 \phi^3$ and we obtain the inner equation
\begin{equation}\label{innernodiscrepant}
\phi(z+1) -2 \phi(z) +\phi(z-1) = -\phi^3(z).
\end{equation}

We can also encounter inner equations having infinite terms in its right hand side. For instance by considering
$f(x,h) = \varepsilon \sin(x) + \sum_{k\geq 1} a_{k} h^{2k+2} \sin\big ( (k+1) x \big )$. In this case $d_{2k}=2k+2$, $d_{2k+1}=0$, $n=2$, $\alpha=2$
and $I=\{k\in \N : k\;\; \text{is even}\}$, so that the inner equation is
\begin{equation*}
\phi(z+1) -2 \phi(z) +\phi(z-1) = -\ee^{\phi(z)} + \sum_{k\geq 2} G_k\big (\ee^{k\phi(z)}).
\end{equation*}

The main purpose of this paper is to provide some particular
solutions of the
%\emph{inner equation}~\eqref{innereqpresentation}
\emph{inner equation}~\eqref{innereqpresentation} and~\eqref{innereqpresentation-trig}
as well as to compute an explicit formula for their difference. The
precise statement is placed in next section, while its proof is
spread along the subsequent ones. As we have already commented in
Section~\ref{introduction:innerequation}, this computation has been in the
heart of the proof of the splitting of separatrices in all the known
examples, and it also gives an explanation to the numerical results
concerning the example in Section~\ref{discrepantexample}.
\section{Main results}
\label{sec:mainresults} We consider the linear operators
\begin{equation}\label{Delta1}
\Delta (\phi)(z) = \phi(z+1) - \phi(z)
\end{equation}
and
\begin{equation}\label{Delta2}
\Delta^2 (\phi)(z) = \Delta(\phi)(z) - \Delta(\phi)(z-1) = \phi(z+1)-2\phi(z) + \phi(z-1)
\end{equation}
and two types of
%\emph{inner equation},
inner equation.
The first one, under the hypotheses \textbf{HP1}, \textbf{HP2}, which from now on we will call polynomial case,
\begin{equation}\label{innereq}
\Delta^2 (\phi) = g(\phi,\p):=-\phi^n + \G(\phi,\p)
\end{equation}
and the second one, under the hypotheses \textbf{HT1}, \textbf{HT2}, which we will call trigonometric case,
\begin{equation}\label{innereqt}
\Delta^2 (\phi) = g(\phi,\p):=-\ee^{(n-1)\phi} + \G(\ee^{\phi},\p),
\end{equation}
 with $\G$ an analytic function in some open bidisk
$\mathbb{D}(\rF)\times \mathbb{D}(\p_0) \in \mathbb{C}^2$ and such
that
\begin{align}
\label{def:F} \G(y,\p) &= \sum_{k\geq n+1} \G_k(\p) y^{k}  & \text{in the polynomial case}\\
\label{def:Ftrig} \G(\ee^{y},\p) &= \sum_{k\geq n} \G_k(\p) \ee^{ky} & \text{in the trigonometric case}.
\end{align}
The parameter~$\p$ is included for the sake of completeness and it is a regular parameter.
\begin{remark}{\rm
Let $\alpha\in \mathbb{R}$ be such that $\alpha  n>1$.
If we consider inner equations of the form either $\Delta^2(\phi) = g(\phi^{\alpha},\p)$ in the polynomial case or
$\Delta^2(\phi) = g(\alpha \phi,\p)$ in the trigonometric one, the results in this section also hold true with the
same proof. However, in order to avoid a new parameter, we restrict ourselves to the hypotheses above.}
\end{remark}

In this section we present the results dealing with both formal and
analytic solutions of the
%\emph{inner equation}.
inner equation.

Given $\nu>0$, we will denote by
$$
\C[[z^{-\nu}],\{\p\}] = \left \{ \phi(z) = \sum_{k\geq 1}
\frac{c_{k-1}(\p)}{z^{\nu k }} \;\;\vert \;\; c_{k-1}:B(\p_0) \to \C
\right \}
$$
the space of formal power series in $z^{-\nu}$ without constant
term, whose coefficients $c_{k-1}$, depend analytically on $\p \in
B(\p_0)$.

\begin{proposition}
\label{prop:formalsolutionsoftheinnerequation} Let $n\ge 2$, $r=
2/(n-1)$.
\begin{enumerate}
\item If $n$ is even, the equations~\eqref{innereq}
and~\eqref{innereqt} admit a unique formal solution $\tilde \phi$
such that $\tilde \phi \in \C[[z^{-r}],\{\p\}]$ with $c_0^{n-1} =
-r(r+1)$, in the case of~\eqref{innereq}, and, in the case
of~\eqref{innereqt}, $\tilde \phi - \tilde \phi_0 \in
\C[[z^{-r}],\{\p\}]$, with
\begin{equation}\label{deftildephi0formal}
 \tilde \phi_0  (z) =
\frac{1}{n-1}\log\left(-\frac{2}{n-1}\frac{1}{z^2}\right).
\end{equation}
Moreover, any formal solution of the
%\emph{inner equation}~\eqref{innereq}
inner equation~\eqref{innereq}
belonging to $\C[[z^{-r/2}],\{\p\}]$ is of the form $\tilde \phi
(z-c,\p)$, for some  $c\in \C$ and $c_0$ such that $c_0^{n-1} =
-r(r+1)$. The same applies to any formal solution~$\phi$
of~\eqref{innereqt} such that $\phi-\tilde
\phi_0\in\C[[z^{-r/2}],\{\p\}]$.
\item  If $n = 2m -1$ with $m\ge 2$, the formal solutions are
\begin{equation}\label{formalsolutionnodd}
\tilde \phi  (z,\p)=
%\sum_{k=1}^{m-1}\frac{c_{k-1}(\p)}{z^{kr}}
% \sum_{l \ge 0} \sum_{ k = l(m-1) +1}^{(l+1)(m-1)}
% \sum_{j = 0}^{l} c_{k-1,j} \frac{\log^j
% z}{z^{kr}} =
%+
\sum_{k\ge 1} \frac{1}{z^{kr}}\sum_{0\le j \le
\left[\frac{k-1}{m-1}\right]} c_{k-1,j}(\p) \log^j z,
\end{equation}
in the case of~\eqref{innereq}, with $c_0 = c_{0,0}$ satisfying
$c_0^{n-1} = -r(r+1)$, and
\begin{equation}\label{formalsolutionnoddt}
\tilde \phi  (z,\p)=
\frac{1}{n-1}\log\left(-\frac{2}{n-1}\frac{1}{z^2}\right)
% \sum_{l \ge 0} \sum_{ k = l(m-1) +1}^{(l+1)(m-1)}
% \sum_{j = 0}^{l} c_{k-1,j} \frac{\log^j
% z}{z^{kr}} =
+\sum_{k\ge 1} \frac{1}{z^{kr}}\sum_{0\le j \le
\left[\frac{k}{m-1}\right]} c_{k-1,j}(\p) \log^j z,
\end{equation}
in the case of~\eqref{innereqt}. The symbol $[x]$ stands for the
integer part of $x$. The coefficients~$c_{k-1,j}$ are analytic
functions in $B(\p_0)$.

The solution is unique provided that $c_{m-1,0} = 0$. Any other
formal solution of the form~\eqref{formalsolutionnodd}
or~\eqref{formalsolutionnoddt} is obtained from these ones by
translation.
\end{enumerate}
\end{proposition}

Now we deal with the analytic solutions of the
%\emph{inner equation}.
inner equation.
Let us define the complex domains where these solutions are defined.
For any $\r,\g>0$, we introduce
\begin{equation}\label{defDus}
D_{\g,\r}^{\s}=\{z\in \C : |\im z | > -\g \re z + \rho\},\;\;\;\; D_{\g,\r}^{\u} = -D_{\g,\r}^{\s}.
\end{equation}
\begin{figure}[h]
 % \includegraphics[width=12cm]{dominiestable2.jpg}\\
\begin{center}  \includegraphics[width=12cm]{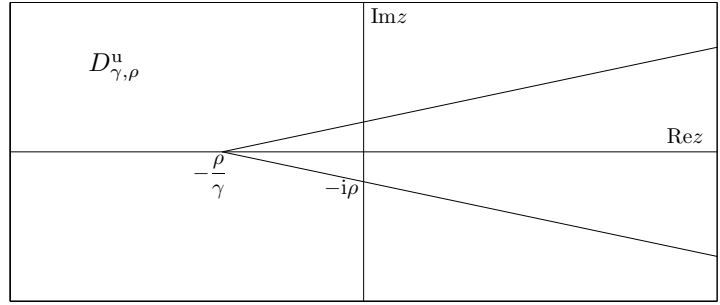}
  \end{center}
    \caption{Unstable domain}\label{dominiestable}
\end{figure}

Let $\tilde{\phi}_0$ be defined by \eqref{deftildephi0formal} in the trigonometric case
and $\tilde{\phi}_0\equiv 0$ in the polynomial case.
Let $\phi_0$ be the truncation up to order $n$ in $z^{-r}$ of the
formal solution provided by
Proposition~\ref{prop:formalsolutionsoftheinnerequation}, that is,
if $n=2m$ with $m\ge 1$,
\begin{equation}
\label{def:phi0even} \phi_0 (z) = \tilde{\phi}_0(z) + \sum_{k=1}^n c_{k-1}(\mu) z^{-kr},
\end{equation}
 and, if $n=2m-1$ with $m\ge 2$, $\phi_0$ in the polynomial case is
\begin{equation}
\label{def:phi0oddp} \phi_0 (z) =
% \sum_{l \ge 0} \sum_{ k = l(m-1) +1}^{(l+1)(m-1)}
% \sum_{j = 0}^{l} c_{k-1,j} \frac{\log^j
% z}{z^{kr}} =
\sum_{k= 1}^n \frac{1}{z^{kr}}\sum_{0\le j \le
\left[\frac{k-1}{m-1}\right]} c_{k-1,j}(\p) \log^j z
\end{equation}
and in the trigonometric case
\begin{equation}
\label{def:phi0oddt} \phi_0 (z) = \tilde{\phi}_0(z)
+\sum_{k= 1}^n \frac{1}{z^{kr}}\sum_{0\le j \le
\left[\frac{k}{m-1}\right]} c_{k-1,j}(\p) \log^j z.
\end{equation}
\begin{theorem}\label{existencetheorem}(Existence theorem)
Let $r=2/(n-1)$ and $c_0$ be such that $c_0^{n-1}=-r(r+1)$. For any
$\g>0$ there exists $\r_0$ big enough such that for any $\r\geq
\r_0$, the
%\emph{inner equations}~\eqref{innereq}
inner equations~\eqref{innereq}
and~\eqref{innereqt} have two analytic solutions
$\phi^{\u,\s}:D_{\g,\r}^{\u,\s} \times B(\p_0)\to \C$ such that
\begin{equation*}
\phi^{\u,\s} (z,\mu) = \phi_0(z)+ \f^{\u,\s}(z,\mu),
\end{equation*}
with
\[
\sup_{(z,\mu) \in D_{\g,\r}^{\u,\s} \times B(\p_0)} |z^{r+2}
\f^{\u,\s}(z,\mu)|<+\infty,
\]
%\begin{align*}
%&\sup_{(z,\mu) \in D_{\g,\r}^{\u,\s} \times B(\p_0)} |z^{2r} \f^{\u,\s}(z,\mu)|<+\infty,\;\;\;\;\;\; \text{if}\;\;\;\;\;\; n\;\; \text{is even} \\
%&\sup_{(z,\mu) \in D_{\g,\r}^{\u,\s} \times B(\p_0)} |(\log z)^{-1} z^{2r} \f^{\u,\s}(z,\mu)|<+\infty,\;\;\;\;\;\; \text{if}\;\;\;\;\;\; n\;\; \text{is odd}
%\end{align*}
\end{theorem}

Now we state the theorem for the difference $\phi^{\u}-\phi^{s}$. First we define the complex domain (see Figure \ref{dominiiner})
\begin{equation}\label{defEdomain}
E_{\g,\r}=D_{\g,\r}^{\u} \cap D_{\g,\r}^{\s} \cap \{ z\in \C:  \im z <0\} \backslash\{z\in \C: |\re z|\leq 1, \; |\im z|\leq \r+\g\},
\end{equation}
where the difference between two solutions of the
%\emph{inner equation}~\eqref{innereq}
inner equation~\eqref{innereq}
$\phi^{\u}-\phi^{\s}$ is defined.

\begin{figure}[h]
\begin{center}
  \includegraphics[width=9cm]{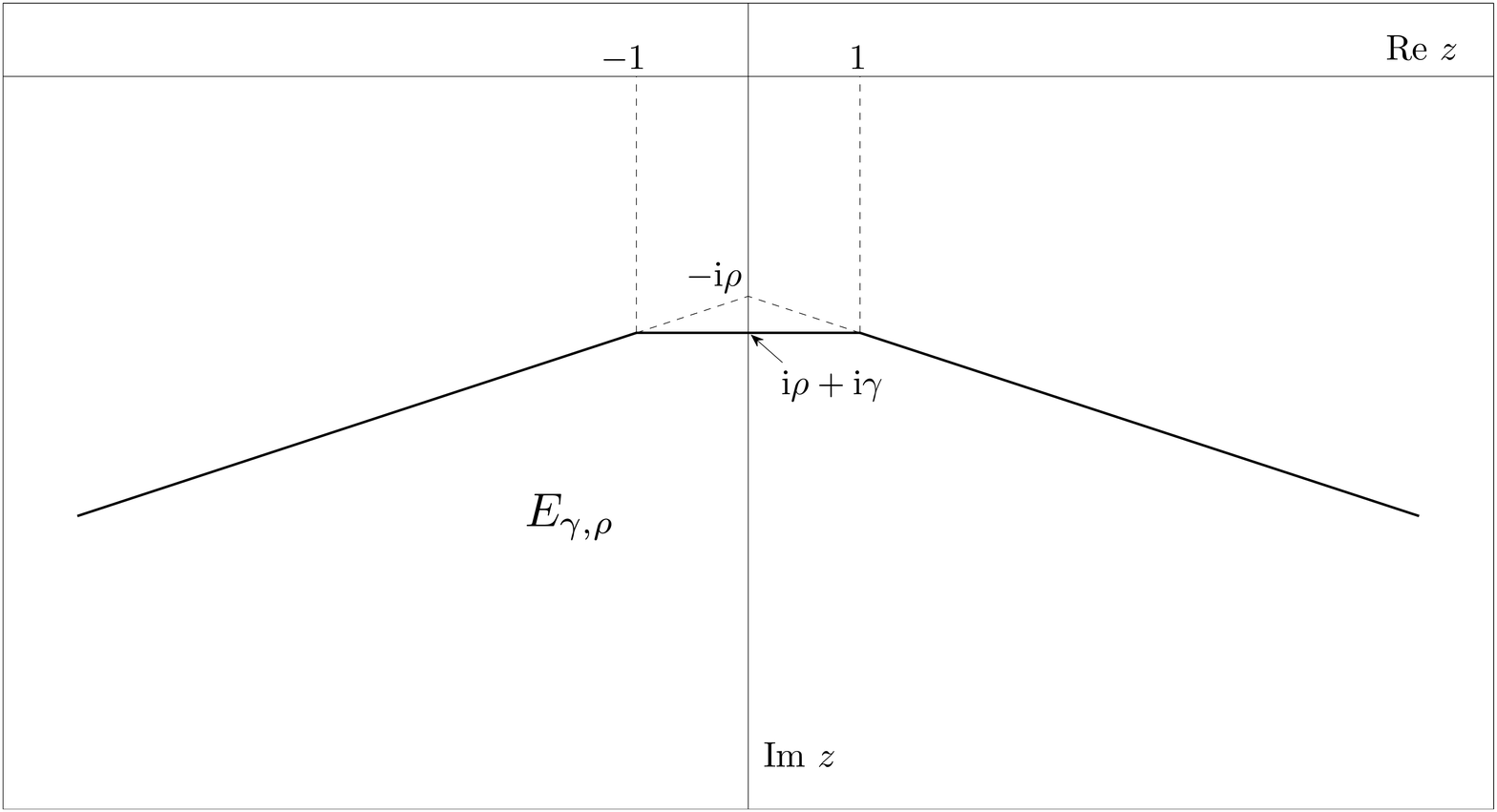}
  \end{center}
  \caption{Inner Domain}\label{dominiiner}
\end{figure}

To unify the notation we introduce the new parameters
\begin{equation}\label{defeH}
\eH=\left \{ \begin{array} {cl} r+2 & \text{polynomial case} \\ 2 & \text{trigonometric case} \end{array} \right.,
\;\;\;\;\;d_{\eH} =\left \{ \begin{array}{cl} c_0  & \text{polynomial case} \\ 1 & \text{trigonometric case} \end{array} \right.
\end{equation}
%\begin{align}
%\eH&=\left \{ \begin{array} {cl} r+2 & \text{polynomial case} \\ 2 & \text{trigonometric case.} \end{array} \right. \label{defeH}\\
%d_{\eH} &=\left \{ \begin{array}{cl} c_0  & \text{polynomial case} \\ 1 & \text{trigonometric case.} \end{array} \right. \label{defdeH}
%\end{align}
%We note that $\partial_z \phi^{\u,\s} = -r d_{\eH} z^{-(\eH-1)} + \gen_{}$.

\begin{theorem}\label{differencetheorem} Let $\phi^{\u,\s}$ be two analytic solutions of equations~\eqref{innereq}
and~\eqref{innereqt} satisfying the conditions stated in
Theorem~\ref{existencetheorem}.

Their difference $\phi^{\u}-\phi^{\s}:E_{\g,\r}\times B(\p_0) \to \C$, can be expressed as
\begin{equation}\label{exp:theorem:difference}
\phi^{\u}(z,\p)-\phi^{\s}(z,\p) =  \spli_1(z,\mu)\sum_{k<0}  p_k^1(\p)\ee^{2\pi \ic kz} + \spli_2(z,\mu) \sum_{k<0} p_k^2 (\p)\ee^{2\pi \ic kz}
\end{equation}
with $p_k^1, p_k^2$ analytic functions in $B(\p_0)$ and $\spli_1,\spli_2$ satisfying that:
\begin{enumerate}
\item their wronskian
$$
W(\spli_1,\spli_2 ):=\left | \begin{array}{cc} \spli_1(z,\p) & \spli_2(z,\p) \\ \spli_1(z+1,\p) & \spli_2(z+1,\p) \end{array}\right | =1
$$
\item there exists a constant $C$ such that for any $z\in E_{\g,\r}$ and $\p \in B(\p_0)$,
$$
\big \vert z^{1-\eH} \ee^{2\pi \ic z} \big (\spli_1(z,\p) - \partial_z \phi^{\s}(z,\p)\big) \big \vert \leq C , \;\;\;\;\;
\left \vert\frac{z^{-\nu}}{\log^{\sigma} z}\left (\spli_2(z,\p) - \frac{z^{\eH}}{r d_{\eH}(2\eH-1)}\right )\right \vert \leq C
$$
with $\nu=\eH-r$, $\sigma=0$ if $n>3$, $\nu=\eH-1$ if $n\leq 3$, $\sigma=0$ if $n=2$ and $\sigma=1$ if $n=3$.
\end{enumerate}
\end{theorem}

From now on we will skip the dependence on $\mu$ being always analytic.

\section{Numerical results}
\label{numericalresults}

In this section we present some numerical results concerning the
generalized standard map~\eqref{def:generilizedstandardmaps} given
by the functions $\ff_1(x,h)=\varepsilon (x-x^3) - \varepsilon^2 x^7$ in~\eqref{gsm7}
and $\ff_2(x,h)= \varepsilon(x-x^3)$. We recall here that $\varepsilon = 4 \sinh^2(h/2)$.

We notice that both functions $\ff_1, \ff_2$ satisfy the hypotheses of
Section~\ref{subsec:innerequation}.  Henceforth, as we show in Section~\ref{Examplesinner}, we can construct the
%\emph{inner equation}
inner equation for the generalized standard map induced by them:
\begin{equation}
\label{eq:innerequation7} \Delta^2(\phi) = - \phi^7\;\;\;\;\; \text{and}\;\;\;\;\; \Delta^2 (\phi) = -\phi^3.
\end{equation}
The first one corresponds to $\ff_1$ and the second one to $\ff_2$.

Let $$\difin := \phi^{\u}-\phi^{\s},$$ be the difference between the
two solutions of the
%\emph{inner equation}~\eqref{eq:innerequation7}
inner equation~\eqref{eq:innerequation7}
given by Theorem~\ref{existencetheorem}. First of all, in a general
setting, we relate the main term of $\difin$ with the Lazutkin
invariant for the standard map \eqref{def:generilizedstandardmaps}
induced by $\ff$. Next, we compute the actual Lazutkin invariant for
the maps defined by $\ff_1$ and $\ff_2$ which is computed
numerically by using multiprecision routines. After we summarize the
method to compute the main term of the difference $\difin :=
\phi^{\u}-\phi^{\s},$ by exploiting the theoretical framework we
have developed. One aspect worth to remark is that these
computations have been performed through standard {\tt long double}
precision arithmetic.

A similar, but more detailed, numerical comparison between the
Lazutkin invariant and the difference $\difin$ is performed
in~\cite{GG11} for the Swift-Hohenberg equation.

\subsection{The relation between the Lazutkin invariant and $\difin$}
For computing the first
asymptotic term of $\difin$ we now take advantage from the fact that
we have an alternative expression for $\difin$ by using the
functions $\spli_{1}$ and $\spli_{2}$ given in
Theorem~\ref{differencetheorem}. Indeed, we actually can write the
difference $\difin$ as
\begin{equation*}
\difin(z)=  \spli_1(z)\sum_{k<0}p_k^1
\ee^{2\pi \ic kz} + \spli_2(z) \sum_{k<0} p_k^2 \ee^{2\pi
\ic kz}
\end{equation*}
with $p_j(z)=\sum_{k<0} p_k^j\ee^{2\pi \ic kz}$, $j=1,2$ $1$-periodic functions.
We recall that by Theorem \ref{differencetheorem}
$W(\spli_1,\spli_2)=1$ and henceforth $p_1=W(\difin,\spli_2)$ and $p_2 = W(\spli_1,\difin)$.

On the one hand, we introduce the new quantity $\omega_{{\rm in}}(z)$:
\begin{equation}\label{innerinvLazutkin2}
\omega_{{\rm in}}(z):= -\frac{d}{dz} W(\difin,\spli_1)(z) =\frac{d}{dz} p_2(z) =\sum_{k<0} 2\pi \ic k p_k^2\ee^{2\pi \ic kz} \thickapprox
-2\pi \ic p_{-1}^2 \ee^{-2\pi \ic z}.
\end{equation}
The last equality has been deduced as $\im z \to -\infty$. On the other hand, note that by using the first approximations of $\spli_1$ and $\spli_2$ in Theorem \ref{differencetheorem},
since $\eH>0$, $\spli_1(z) \to 0$ as $\im z \to -\infty$ and $\spli_2(z) =\gen(z^{\eH})$,
the main term of $\difin$ is
\begin{equation*}
\difin(z)=\spli_1(z) p_1(z) + \spli_2(z)p_2(z) \thickapprox \spli_2(z) p_2 (z)\thickapprox z^{\eH}\frac{p_{-1}^2}{r d_{\eH}(2\eH-1)} \ee^{-2\pi i z}.
\end{equation*}
We recall here that only $p_{-1}^2$ is unknown, the other quantities are defined in terms of the inner equation.
Henceforth, both $\omega_{{\rm in}}(z)$ and $z^{-\eH}\difin(z)$, are asymptotically equivalent.

In order to compare the numerical results with our theoretical
framework we will gather in a rather informal way several facts,
some of them not proven. In particular, to transform assumptions (A1) and (A2) below
into proven facts would require involved arguments even for particular cases.
By this reason, we will avoid precise statements. The chain of reasoning is a slight modification
of the one in~\cite{MSS11a}, which also follows~\cite{Laz84,Gel99}.

Let $f$ be a real analytic function satisfying the hypotheses in Section~\ref{firstasumptions}
and Section~\ref{subsec:innerequation}.
We first remark that there exists a solution of the invariance
equation~\eqref{eq:invariance} induced by $f$, $x^{\u}(t)$,  $\ic \pi$-antiperiodic,
entire and real analytic in $t$, such that $\lim_{\re t \to -\infty}
x^{\u}(t) = 0$ and $x^{\u}(0)=x^{\u}(-h)$ (and $x^{\u}(t)-x^{\u}(t-h) >0$, for
$t\le 0$).
%\footnote{This claim follows easily from the literature.
%For instance, the existence of and $i\pi$~antiperiodic
%parametrization of the unstable invariant manifold of the origin
%defined for $\re t < R$ follows from Birkhoff normal form. Since the
%map is entire, this parametrization is entire. A suitable traslation
%in $t$ provides the desired solution.}
Then, the function $x^{\s}(t) = x^{\u}(-t)$ is also a solution
of~\eqref{eq:invariance}, with the same regularity, satisfying
$\lim_{\re t \to \infty} x^{\s}(t) = 0$. Hence, $\gamma^{\u,\s}(t) =
(x^{\u,\s}(t),x^{\u,\s}(t)-x^{\u,\s}(t-h))$ are natural parametrizations
of the invariant manifolds of the origin. We notice that $p = \gamma^{\u}(0) = \gamma^{\s}(0) =
(x^{\u,\s}(0),0)$ is the \emph{first} homoclinic point. Let
$D(t) = x^{\s}(t)-x^{\u}(t)$.

Using the $h$-step Wronskian
\[
W_h(u,v)(t) = \begin{vmatrix} u(t) & v(t) \\
u(t)-u(t-h) & v(t)-v(t-h)
\end{vmatrix} = \begin{vmatrix} u(t) & v(t) \\
\Delta_h u(t) & \Delta_h v(t)
\end{vmatrix}
\]
the Lazutkin invariant~\eqref{def:Lazutkininvariant} can be written
as
\begin{equation}
\label{eq:lazutkininvariantwronskian} \omega(p) = \det(\dot
\gamma^u, \dot \gamma^s)_{\mid t=0} = \frac{d}{dt} \det(\dot
\gamma^u, \gamma^s-\gamma^u)_{\mid t=0}= \frac{d}{dt} W_h(\dot x^u,
D)_{\mid t=0}.
\end{equation}

Since both $x^u$ and $x^s$ are solutions of the second order
difference equation~\eqref{eq:invariance}, their difference $D$ also
satisfies a \emph{linear} second order equation, namely
\begin{equation}
\label{eq:D} \Delta_h^2 D(t) = -\left (\int_0^1  \frac{\partial}{\partial x}\ff(s x^{\s}(t)+(1-s)x^{\u}(t) , h) \,ds \right ) D(t).
\end{equation}
Notice that, if $x^{\u}$ is close to $x^{\s}$, then
equation~\eqref{eq:D} is close to the linearization of the
invariance equation~\eqref{eq:invariance} around $x^{\u}$. Hence, our
first assumption is that
\begin{enumerate}
\item[(A1)]
there is a (real analytic) solution $\eta_1$ of
equation~\eqref{eq:D} close to~$\dot x^{\u}$.
\end{enumerate}
Let $\eta_2$ be another (real analytic) solution of~\eqref{eq:D}
with $W_h(\eta_1,\eta_2) = 1$, which can be obtained by the
``variation of constants'' method. Hence, we can write $D =
c_1\eta_1+c_2\eta_2$ where $c_1$ and $c_2$ are the $h$-periodic
functions
$c_1 = W_h(D,\eta_2)$ and $c_2 =
W_h(\eta_1,D)$. Substituting this expression for $D$
into~\eqref{eq:lazutkininvariantwronskian} and using that $\eta_1$
is close to $\dot x^u$ we have that
\begin{equation}
\label{eq:Lazutkininvariantsecondformula} \omega(p) \thickapprox
\frac{d}{dt} W_{h}(\eta_1,D)_{\mid t=0}.
\end{equation}

Since $\ff$ satisfies the hypotheses of Section \ref{subsec:innerequation},
we can construct an
%\emph{inner equation}
inner equation
associated to the standard map induced by$f$. The second assumption is
\begin{enumerate}
\item[(A2)] there exists $\chi \in \C$ (which can depend on $h$) such that
for values of $t$ satisfying $|t-\chi| = \gen(h)$, $x^{\u,\s}(t)$ are
close to $h^{-\alpha} \lambda \phi^{\u,\s}((t-\chi)/h)$. Here $\phi^{\u,\s}$ are the solutions of the
%\emph{inner equation}~\eqref{innereq}
inner equation~\eqref{innereq}
given by Theorem~\ref{existencetheorem} and $\alpha, \lambda$
are both parameters introduced in Section \ref{subsec:innerequation}.
Since $\ff$ is real analytic one can assume that $\im \chi >0$.
\end{enumerate}
As a consequence, since, by
Theorem~\ref{differencetheorem}, $\spli_1(z) = \partial_z
\phi^{\u}(z) + \gen(z^{r+1} \ee^{2\pi \ic z})$,
\[\dot x^{\u}(t) \thickapprox h^{-\alpha -1} \lambda \frac{d}{dz}
\phi^{\u}((t-\chi)/h)\thickapprox h^{-\alpha-1} \lambda
\spli_1((t-\chi )/h).
\]
Recall now that $
p_2(z) =  -W(\difin,\spli_1)(z)
$.
Hence, taking into account the scaling and assumption~(A1), for values of $t$ close to
$\chi$,
\begin{equation*}
W_h(\eta_1, D) (t) \thickapprox W_h( h^{-\alpha-1}\lambda \spli_1,
h^{-\alpha} \lambda\difin)((t-\chi)/h) = h^{-2\alpha -1} \lambda^2 p_2\big ((t-\chi)/h \big).
%\sum_{k<0}p_k^2 \ee^{2\pi \ic k(t-\chi)/h}.
\end{equation*}
Then, since $W_h(\eta_1, D)(t)$ and
$W(\spli_1, \difin)((t-\chi)/h\big )$ are both $h$-periodic and
that the first one is a real analytic function
we easily have that, for real $t$
\[
\frac{d}{dt} W_h(\dot x^{\u},D)(t) \thickapprox 2 h^{-2\alpha-2}
 \re \left (\lambda^2 \cdot \frac{d}{dz} p_2\big((t-\chi)/h\big )\right )
%\sum_{k<0}4\pi \ic k \re \big (p_k^2 \ee^{2\pi \ic k (t-\chi)/h}\big )
= 2 h^{-2\alpha -2}  \re \left (\lambda^2 \cdot \omega_{{\rm in}}\big ((t-\chi)/h\big )\right )
\]
with $\omega_{{\rm in}}$ defined in~\eqref{innerinvLazutkin2}.
Hence, evaluating at $t=0$,
\begin{equation}\label{innerinvLazutkin1}
\omega(p) \thickapprox 2 h^{-2\alpha-2}  \re \left (\lambda^2 \cdot\omega_{{\rm in}}(-\chi/h)\right )
%\Leftrightarrow \omega(p) = \lim_{h\to 0} 2 h^{-2\alpha-2} \lambda^2  \re \left (\omega_{{\rm in}}(-\chi/h)\right ).
%\thickapprox -4 \pi h^{-2\alpha-2} \lambda^2 \ee^{-2\pi \im \chi/h} \re \big (\ic p_{-1}^2 \ee^{2\pi \ic \re \chi/h}\big )
.
\end{equation}
Our goal now is to check numerically the above formula for the maps induced by $\ff_1$ and $\ff_2$.

\subsection{The \emph{limit flow} and its singularities}\label{section:limitflow}
In the cases of the Chirikov standard map and the perturbations of
the McMillan map in~\cite{Gel99} and~\cite{MSS11a}, resp, $\chi =
\ic \pi/2$ is the closest to the real line singularity of the
homoclinic orbit of the limit flow~\eqref{def:limitfloweq2}. In the
maps induced by $\ff_1(x,h) =  \e (x - x^3) - \e^2 x^7$ and
$\ff_2(x,h)= \e (x- x^3)$ under consideration, the closest to the
real line singularity of the homoclinic of the limit flow $\ddot{x}
= x-x^3$ is also $\ic \pi/2$ (see Section \ref{discrepantexample}).
Nevertheless, our numerical computations show that it is not the
right guess for~$\chi$ in the case of~$\ff_1$. For this reason, we
consider the \emph{higher order} (in $h$) limit flow
\begin{equation}
\label{eq:alternatelimitflow}
 \ddot x = x - x^3 - \e x^7.
\end{equation}
The parametrization, $x_0(t,h)$, of the homoclinic loop to the
origin such that $\dot x_0(0,h) = 0$ has a singularity at
\[
\rho(h) = \int_{x_0(0,h)}^{+\infty} \frac{dx}{\sqrt{x^2/2 -
x^4/4 - \e x^8/8}},
\]
where $x_0(0,h) = \sqrt{2} + O(h^2)$ is the positive root of $x^2/2
- x^4/4 - \e x^8/8$ and the integral is computed along the real
line.  The other singularities can be obtained changing the path of
integration. It can be seen that
\begin{equation}
\label{chiexpansion} \rho(h) = \ic \frac{\pi}{2} - \ic
\frac{2^{1/4} \Gamma(3/4)^2}{\sqrt{\pi}} h^{1/2} + O(h^{3/2}).
\end{equation}
We remark that, although the singularities of the homoclinic
of~\eqref{eq:alternatelimitflow} tend to the singularities of the
limit flow $\ddot{x}=x-x^3$~\eqref{def:limitfloweq2} (in a rather slow way), they are
of a different type: whereas the latter are poles, the former are
branching points.

We choose the values
$\chi = \ic \frac{\pi}{2} - \ic
\frac{2^{1/4} \Gamma(3/4)^2}{\sqrt{\pi}} h^{1/2}$ for $\ff_1$ and $\chi = \ic \frac{\pi}{2}$ for
$\ff_2$ and we will assume that (i) holds for them.

\subsection{Numerical computations}
We define now
\begin{equation}\label{definiciontildeomega}
\tilde{\omega}(h)= h^{2\alpha + 2} \lambda^{-2}\ee^{2\pi |\chi|/h}\omega(p),\;\;\;\;\;\;
%\tilde{\omega}_{{\rm in}} (h)= 2 \ee^{2\pi |\chi|/h} \re \left (\omega_{{\rm in}} \big(-\chi/h\big)\right )
\tilde{\omega}_{{\rm in}} (z)= 2 \ee^{2\pi \ic z} \re \left (\omega_{{\rm in}}(z)\right )
\end{equation}
taking $\lambda=1$ and on the one hand
$\alpha= 2/3$ for $\ff_1$ and on
the other hand $\alpha=1$ for
$\ff_2$. We note that, since $\chi$ has no real part, checking formula \eqref{innerinvLazutkin1}
is equivalent to check that
\begin{equation*}
\tilde{\omega}(h) \thickapprox \tilde{\omega}_{{\rm in}}\big (-\chi /h\big) \Leftrightarrow \lim_{h\to 0} \tilde{\omega}(h) = \lim_{\im z \to -\infty} \tilde{\omega}_{{\rm in}}(z).
\end{equation*}

First we show the results for $\tilde{\omega}(h)$.
We have computed numerically this quantity by
using multiprecision routines written in PARI-GP. In the following figure we show the computed values
for $\ff_1(x,h)=\e(x-x^3)-\e^2 x^7$ and for
the map induced by $\ff_2(x,h) = \varepsilon (x-x^3)$. Let us denote by $\tilde \omega_i(h)$ the value of $\tilde \omega(h)$ for
the corresponding maps $f_i$, $i=1,2$.
We have added a correction factor
${\rm esc} = 85 \cdot 10^{-4}$ in order to have the same magnitude for both values of $\tilde{\omega}(h)$.
\begin{figure}[h]
 % \includegraphics[width=12cm]{dominiestable2.jpg}\\
\begin{center}  \includegraphics[width=12cm]{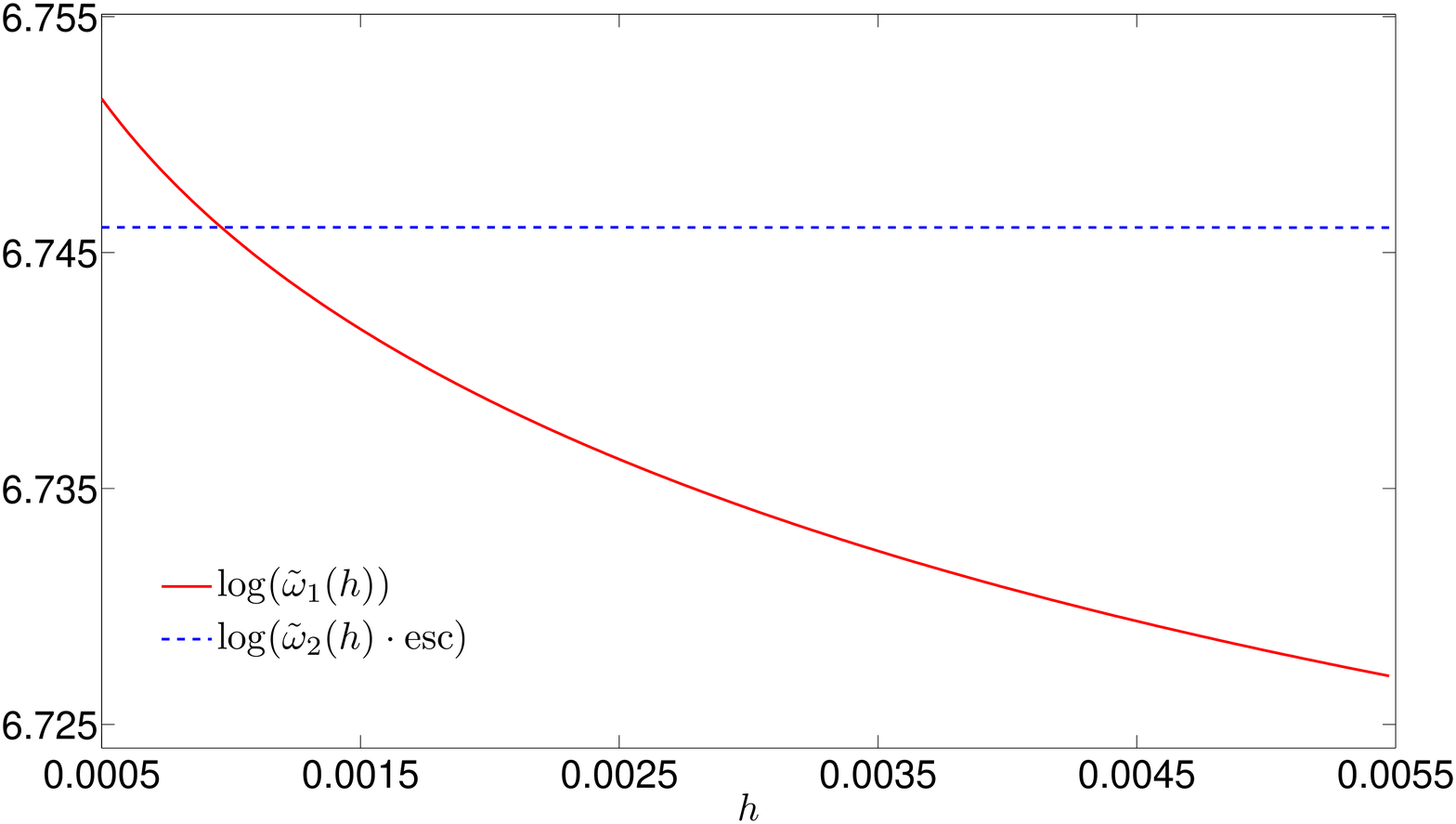}
  \end{center}
    %\caption{$\log(\tilde{\omega}(h))$}
    \label{figureomegatilde}
\end{figure}

These numbers have been
obtained computing explicitly $\omega(p) = \det({\dot
\gamma}^u(0),\dot \gamma^s(0))$, following the strategy
in~\cite{DR99}. Due to the exponentially small behavior of this
quantity, it has been necessary to compute $\dot \gamma^{u,s}(0)$
with increasing accuracy, thus making impossible to achieve very
small values of $h$.

Notice that, in the case of the map induced by $\ff_2(x,h) = \varepsilon
(x-x^3)$, the values of $\tilde \omega_2(h)$ converge quite fast, when
$h$ becomes smaller, to a constant value
\begin{equation}\label{tildeomegabuenon3}
\tilde{\omega}_2(h) \thickapprox 1.00083 \cdot 10^5
\end{equation}
In the case of the map induced by $\ff_1(x,h) = \e(x-x^3)-\e^2 x^7$, the
convergence of the values of $\tilde \omega_1(h)$ is slower, as the
previous figure shows. However, computing $\tilde \omega_1(h)$ for
$h=1/2000+k/40000$, $k=0\dots 199$, and making some assumptions on
the form of the asymptotic expansion of $\tilde \omega_1(h)$ in~$h$,
it is possible to extrapolate the limit value with better accuracy.

In this way, we have obtained that it is
\begin{equation}\label{tildeomegabuenon7}
\tilde{\omega}_1(h)\thickapprox 871.683
\end{equation}
We remark that, with the computed data, in which each value of
$\tilde \omega_{i}(h)$, $i=1,2$, has a few hundreds of correct digits, it would
be possible to obtain a better approximation of this value, and also
to compute the coefficients of the asymptotic expansion. Since our
intention was to compare the results obtained by the analysis of the
solutions of the
%\emph{inner equation},
inner equation,
we have not pursued in this direction.

Now we compute $\tilde{\omega}_{{\rm in}}(z)$
%by means of the approximation $\tilde{\omega}_{{\rm in}}(z) \thickapprox 4\pi \im p_{-1}^2$.
. By definition \eqref{innerinvLazutkin2} of $\omega_{{\rm in}}(z)$ and \eqref{Delta1} of the operator $\Delta$,
\begin{equation}\label{formulatildeomegain}
\tilde{\omega}_{{\rm in}}(z)
 = 4\pi \ee^{2\pi \ic z} \re \big(\difin(z) \cdot \Delta (\partial_z \phi^{\s})(z) -\partial_z \phi^{\s}(z)\cdot \Delta (\difin)(z)\big )
+\mathcal{O}(\ee^{-2\pi \ic z}z^{2r+2},\ee^{-2\pi \ic z}),
\end{equation}
where we have used that by Theorem \ref{differencetheorem}, $\spli_1(z) - \partial_z \phi^{\u}(z) = \gen ( \ee^{-2\pi \ic z} z^{r+1})$.
%The operator $\Delta$ was defined in \eqref{Delta1}.

For symmetry reasons, we choose $z=-i\r$ with $\r\in [2.25,7]$. We have used
{\tt long double} precision in {\tt C} for calculating $\phi^{\s,\u}(z),  \partial_z\phi^{\s,\u}(z) $ . The strategy was suggested in \cite{GL01}
\begin{itemize}
\item First we compute the formal series $\widetilde{\phi}_N$ up to order $N$ big enough. We know that the solutions $\phi^{\s,\u}$ are close
to $\widetilde{\phi}_N$ if $|z|$ is big enough. Analogously for $\partial_z \phi^{\s,\u}$.
\item We evaluate the formal series $\widetilde{\phi}_N(z\pm k)$ and $\partial_z \widetilde{\phi}_N(z\pm k)$ with $k\in \mathbb{N}$ big enough.
\item Since both $\phi^{\s,\u}$ satisfy the
%\emph{inner equation}
inner equation, we obtain $\phi^{\s,\u}(z)$ and $\phi^{\s,\u}(z+1)$ recurrently.
Analogously for
%Differentiating the
%\emph{inner equation}
%inner equation
%we compute
$\partial_z
\phi^{\s,\u}(z)$ and $\partial_z \phi^{\s,\u}(z+1)$
%
%with respect to $z$
%we obtain the recurrence relation for computing $\partial_z
%\phi^{\s,\u}(z)$ and $\partial_z \phi^{\s,\u}(z+1)$.
\end{itemize}

We have computed $\tilde{\omega}_{{\rm in}}(z)$
for the
%\emph{inner equations}~\eqref{eq:innerequation7}.
inner equations~\eqref{eq:innerequation7}.
Our results are given in the following picture,
where we have added the scaling factor ${\rm esc} = 871 \cdot 10^{-5}$.

\begin{figure}[ht]
 % \includegraphics[width=12cm]{dominiestable2.jpg}\\
\begin{center}  \includegraphics[width=12cm]{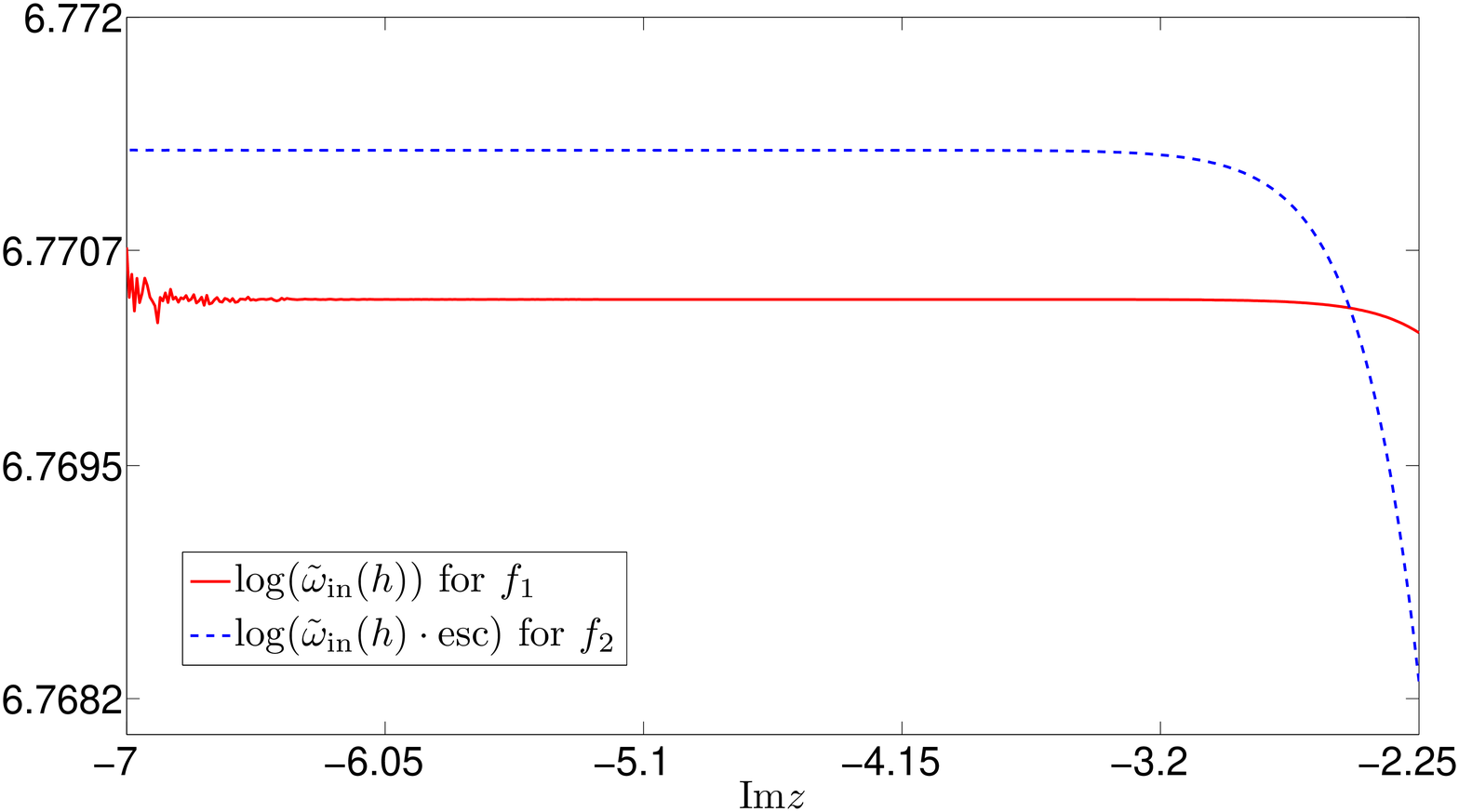}
  \end{center}
    %\caption{$\log(\tilde{\omega}(h))$}
    \label{figureomegainner}
\end{figure}
\begin{comment}
\begin{center}
\begin{tabular}{c|c|c} $z_0$ & $ \Delta^2 (\phi) = -\phi^3 $ & $\Delta^2(\phi) = - \phi^7$\\  \hline& & \\
$-6\ic$ & $ 100083.27465193532 \ic $ & $871.85329382865268 \ic $ \\
$-5\ic$ & $ 100083.27419751925 \ic  $& $871.68272063906670 \ic $ \\
$-4\ic$ & $ 100083.23927877509 \ic $& $871.68329747322048\ic $ \\
$-3\ic$ & $ 100075.95487738834 \ic $& $871.68061127473027\ic $
\end{tabular}
\end{center}
\end{comment}
We can observe that, on the one hand when $\im z \in [-3,-2]$ the theoretical error in \eqref{formulatildeomegain} is big. On the other hand,
when $\im z \in [-7,-6]$ the round-off errors (for $f_1$) begin to be bigger than the theoretical error and hence the computed values have noise.
Nevertheless for values of $\im z\in [-5,-3]$ the computed values of $\tilde{\omega}_{{\rm in}}(z)$ behave like a constant. More precisely,
we have found $\tilde{\omega}_{{\rm in}}(z) =871.6833\dots $ for $\Delta^2 (\phi)=-\phi^7$ and
$\tilde{\omega}_{{\rm in}}(z) =1.000832 \dots 10^5$ for $\Delta^2 (\phi)=-\phi^3$
which agree with the results for $\tilde \omega(h)$ given in \eqref{tildeomegabuenon3}
and \eqref{tildeomegabuenon7}.

\section{Formal solutions of the inner equation}\label{sec:formalsolution}
In this section we prove the existence of formal solutions of the
%\emph{inner equation}~\eqref{innereq}.
inner equation~\eqref{innereq}.
The proof of the existence of formal solutions of
equation~\eqref{innereqt} follows the same procedure. Hence, we skip
it.

We start by defining the spaces which these formal solutions belong
to. For $n\in \mathbb{N}$ and $r(n-1)=2$, we define
\begin{equation}
\label{def:Xr} X_r = \Big\{ \phi(z) = \sum_{k\ge1}
\frac{c_{k-1}}{z^{kr}}|\; c_k \in \C\Big\},
\end{equation}
the space of formal power series in $x^{-r}$ without constant term,
and, if $n = 2m-1$, $m\ge 2$, that is, $r = 1/(m-1)$,
\begin{equation}
\label{def:Xrlog} X_r^{\log} =  \Big\{ \phi(z) = \sum_{k\ge1}
\sum_{0 \le j \le \left[\frac{k-1}{m-1}\right]}
c_{k-1,j}\frac{\log^j z}{z^{kr}}|\; c_{k,j} \in \Big\},
\end{equation}
where $[x]$ denotes the integer part of $x$, the space of formal
power series in $x^{-r}$ and $\log z$, with the power of $\log z$
bounded by the power of $x^{-r}$, without constant term.

We will say that $\phi=\gen_{kr}$, with $k\in \N$, if and only if
$z^{kr} \phi \in \C[[z^{-r}]]$ is a power series with terms
$z^{-jr}$ for $j\geq 0$. We will also use $\gen_{kr,j}$ in
$X_r^{\log}$, with $k\in \N$, $j\in \N\cup\{0\}$, meaning that $\phi
= \gen_{kr,j}$ implies that $ z^{kr} (\log z)^{-j} \phi(z) $ is a
formal power series with terms of the form $z^{-k'r}\log^{j'} z $,
with $k'\ge 0$ and $j' \ge -j$ such that $j' \le 0$ whenever $k' =
0$. We keep both notations in order to emphasize that $\gen_{kr}$ is
a series without logarithms, while $\gen_{kr,0}$ is a series whose
leading term does not have logarithms.

We collect several properties of these spaces in the following
lemma, whose proof is straightforward.

\begin{lemma}
\label{lem:XrXrlogdifferencialalgebras} Let $n\ge 2$, $r=2/(n-1)$
and $\gg$ an analytic function around the origin with $\gg(y) = A
y^{\ell}+ \gen(y^{\ell+1})$, for some $\ell \in \N\cup\{0\}$. The
spaces $X_r$, for $n$ even, and $X_r^{\log}$, for $n$ odd, have the
following properties:
\begin{enumerate}
\item
$X_r$ and $X_r^{\log}$ are invariant by the formal differential
operator $\frac{\partial^2}{\partial z^2}$. Furthermore, if $\phi
\in  X_r^{\log}$ (resp. $X_r$), then $\frac{\partial^2}{\partial
z^2}\phi(z) = z^{-2}  \psi(z)$ with $\psi \in X_r^{\log}$ (resp.
$X_r$).
\item
If $\phi(z) = a z^{-r} + \tilde \phi(z)$, with $\tilde \phi =
\gen_{2r,j}$, $0\le j \le [1/(m-1)]$, resp. $\gen_{2r}$, then $\gg(a
z^{-r} + \tilde \phi(z)) = A a^{\ell} z^{-\ell r} + \varphi(z)$, with
$\varphi = \gen_{(\ell+1)r,j}$, resp. $\gen_{(\ell+1)r}$.
\end{enumerate}
Moreover, in the case $n= 2m-1$, $X_r^{\log}$ is also invariant by
translation, that is, if $\phi(z) \in X_r^{\log}$, then $\phi(z-c)
\in X_r^{\log}$, for any $c\in \C$. In the case $n = 2m$, if $\phi
\in X_r$, then $\phi(z-c) \in X_{r/2}$.
\end{lemma}

We recall the function
$
\gg(y) = y^n-\G(y)
$.
We remark that, since the operator $\Delta^2$ can be written
formally as
\begin{equation}
\label{eq:formalDelta2}
 \Delta^2 \phi (z) = 4 \sinh^2 \left( \frac{1}{2}
\frac{\partial}{\partial z} \right) \phi (z) = \left(
\frac{\partial^2}{\partial z^2} + \frac{1}{12}
\frac{\partial^4}{\partial z^4} + \cdots \right) \phi(z),
\end{equation}
(1)~in Lemma~\ref{lem:XrXrlogdifferencialalgebras} implies that the
%\emph{inner equation}~\eqref{innereq}
inner equation~\eqref{innereq}
is well defined in $X_r$ and
$X_r^{\log}$. We introduce
\begin{equation}
\label{def:epsilonhatepsilon} \epsilon(\phi) = \Delta^2 (\phi) -\gg(\phi).
\end{equation}
It is clear that $\epsilon(\phi)(z) = z^{-2}\hat \epsilon(z)$  with
$\hat \epsilon \in X_r^{\log}$ (resp. $X_r$).

Next lemma follows directly from the definition of $X_r^{\log}$.
\begin{lemma}
\label{lem:errorterm} Let $n=2m-1$, $r=1/(m-1)$ and $\phi \in X_r^{\log}$ and $\epsilon(\phi)$
as in~\eqref{def:epsilonhatepsilon}. If $z^2 \epsilon(\phi) (z)$ has
no terms of order $N$ or smaller in $z^{-r}$ (that is, no terms of
the form $z^{-kr}\log^j z$, with $1\le k \le N$), then
$\epsilon(\phi) = \gen_{(N+1)r+2,L}$, where $L=[N/(m-1)]$.
\end{lemma}
\begin{definition}\label{def:truncated series}
Let $n\ge 2$, $N\in \N$, $r=2/(n-1)$ and $\phi \in X_{r}$ or
$X_{r}^{\log}$. We will call truncated series of order $N$ of $\phi$
to $\tilde \phi_N$ having the form:
\begin{enumerate}
\item If $n$ is even,
$$\tilde \phi_{N}(z) = \sum_{k=1}^N \frac{c_{k-1}}{z^{kr}}$$
\item If $n=2m-1$ is odd,
$$
\tilde \phi_{N}(z) = \sum_{k=1}^{m-1}\frac{c_{k-1}}{z^{kr}}+
\sum_{k=m}^{N} \frac{1}{z^{kr}}\sum_{0\le j \le
\left[\frac{k-1}{m-1}\right]} c_{k-1,j} \log^j z.
$$
\end{enumerate}
\end{definition}

Along the proof of Proposition
\ref{prop:formalsolutionsoftheinnerequation}, we will need to
compute several times the formal series $\gg(\phi+\psi)-\gg(\phi)$,
with different $\phi$ and $\psi$. The following lemma, which follows
from the properties in Lemma~\ref{lem:XrXrlogdifferencialalgebras},
summarizes the result.
\begin{comment}
\begin{lemma}
\label{lem:auxiliarylemma} Let $N\geq 2$, $N\in \mathbb{N}$.
\begin{enumerate}
\item If $\tilde \phi_{N-1}(z) =
\sum_{k=1}^{N-1} c_{k-1} z^{-kr}$ and $\psi_N(z) = c_{N-1} z^{-Nr}$,
then
\begin{equation}
\gg(\tilde \phi_{N-1}(z)+\psi_N(z))- \gg(\tilde \phi_{N-1}(z)) = n
\frac{c_0^{n-1}}{z^2}\psi_N(z) + \gen_{(N+1)r+2}.
\end{equation}
\item If $\tilde \phi_{N-1}$ is of the form
$$
\tilde \phi_{N-1}(z) = \sum_{k=1}^{m-1}\frac{c_{k-1}}{z^{kr}}+
\sum_{k=m}^{N-1} \frac{1}{z^{kr}}\sum_{0\le j \le
\left[\frac{k-1}{m-1}\right]} c_{k-1,j} \log^j z.
$$
and $\psi_N (z) = z^{-Nr} \sum_{1 \le j \le [(N-1)/(m-1)]} \log^j
z$, then
\begin{equation}
\gg(\tilde \phi_{N-1}(z)+\psi_N(z))-\gg(\tilde \phi_{N-1}(z)) = n
\frac{c_0^{n-1}}{z^2}\psi_N(z) + \gen_{(N+1)r+2,L},
\end{equation}
where $L =[N/(m-1)]$.
\end{enumerate}
\end{lemma}
\end{comment}
\begin{lemma}\label{lem:auxiliarylemma}
Let $n\ge 2$, $r=2/(n-1)$, $N\geq 2$, $N\in \mathbb{N}$ and $\phi\in
X_r$ or $X_r^{\log}$. We define $\psi_N = \tilde \phi_N -
\tilde\phi_{N-1}$ where $\tilde \phi_{N}$ and $\tilde \phi_{N-1}$
are the truncated series of order $N$ and $N-1$ respectively. We
have that
\begin{enumerate}
\item If $n$ is even,
$$
\gg(\tilde \phi_{N}(z) ) - \gg(\tilde
\phi_{N-1}(z))=-n\frac{c_0^{n-1}}{z^2} \psi_N(z) + \gen_{(N+1)r+2}.
$$
\item If $n=2m-1$ is odd, writing $L=[N/(m-1)]$,
$$
\gg(\tilde \phi_{N}(z) ) - \gg(\tilde
\phi_{N-1}(z))=-n\frac{c_0^{n-1}}{z^2} \psi_N(z) + \gen_{(N+1)r+2,L}.
$$

\end{enumerate}
\end{lemma}

The following proposition implies the existence of formal solution
of the
%\emph{inner equation}~\eqref{innereq}
inner equation~\eqref{innereq}
and henceforth Proposition
\ref{prop:formalsolutionsoftheinnerequation}.
\begin{proposition}
\label{prop:formalsolutionsoftheinnerequationconcreta} Let $n\ge 2$,
$r= 2/(n-1)$ and $c_0$ be such that $c_0^{n-1}=-r(r+1)$. The inner
equation~\eqref{innereq} admits a formal solution $\phi$ with
$z^{r}(\phi(z) - c_0 z^{-r}) \in X_r$, if $n$ is even, and
$z^r(\phi(z) -  c_0 z^{-r} )\in X_{r}^{\log}$ if $n$ is odd.

Let $N\geq 2$ and $\tilde \phi_N$ be the truncated series defined as
in Definition~\ref{def:truncated series}. Writing the truncation
error of order $N$ as
$$
\epsilon_N:=\epsilon(\tilde \phi_N)=\Delta^2 (\tilde \phi_{N}) -\gg(\tilde \phi_{N}),
$$
where $\epsilon$ was defined by \eqref{def:epsilonhatepsilon}, we
have that
\begin{enumerate}
\item If $n\ge 2$ is even $\epsilon_N = \gen_{(N+1)r+2}$.
\item If $n = 2m -1\ge 2$ is odd and $L = [N/(m-1)]$, then
\begin{enumerate}
\item[(i)] if $1\le N \le m-1$, $\epsilon_{N} = \gen_{(N+1)r+2}$,
\item[(ii)] if $m \le N$,  $\epsilon_{N} = \gen_{(N+1)r+2,L}$.
\end{enumerate}
\end{enumerate}
\end{proposition}
\begin{proof}
We deal first with 1). We prove the claim by induction over $N$. We
start by assuming $N=1$. Let $\tilde \phi_1 (z)= c_0 z^{-r}$. By (2)
in Lemma~\ref{lem:XrXrlogdifferencialalgebras} and
using~\eqref{eq:formalDelta2}, we have that
$$
\epsilon_1(z) = \Delta^2 ( \tilde{\phi}_1)(z) -
\gg(\tilde{\phi}_1(z)) = r(r+1) \frac{c_0}{z^{r+2}} +
\frac{c_0^n}{z^{nr}} + \gen_{(n+1)r}.
$$
The claim for $N=1$ follows from the facts that $r = 2/(n-1)$, which
implies $\gen_{(n+1)r} = \gen_{2r+2}$,  and $c_0^{n-1} = -r (r+1)$.

Now we assume the claim for $N-1$, that is, there exist coefficients
$c_k$, $1 \le k \le N-2$ such that $\tilde \phi_{N-1}$ satisfies
\[
\epsilon_{N-1}(z) = \epsilon(\tilde \phi_{N-1})(z) =
\frac{A_{N-1}}{z^{Nr+2}} + \gen_{(N+1)r+2}.
\]
We look for $\tilde
\phi_N (z) = \tilde \phi_{N-1}(z) + c_{N-1} z^{-Nr}$ satisfying the
claim. We have that
$$
\epsilon_N (z)  = \epsilon_{N-1}(z) +
\Delta^2\left(\frac{c_{N-1}}{z^{Nr}}\right) - \gg \left(\tilde
\phi_{N-1}(z) + \frac{c_{N-1}}{z^{Nr}}\right) + \gg(\tilde
\phi_{N-1}(z))
$$
By 1) of Lemma~\ref{lem:auxiliarylemma},
\begin{equation}
\gg \left(\tilde \phi_{N-1}(z) + \frac{c_{N-1}}{z^{Nr}}\right) -
\gg(\tilde \phi_{N-1}(z)) = -n \frac{c_0^{n-1}}{z^2}
\frac{c_{N-1}}{z^{Nr}} + \gen_{(N+1)r+2}.
\end{equation}
Hence, using again~\eqref{eq:formalDelta2},
$$
\epsilon_N (z)  = \frac{A_{N-1}}{z^{Nr+2}} + \lambda_N \frac{c_{N-1}}{z^{Nr+2}}+
\gen_{(N+1)r +2},
$$
where the coefficient $\lambda_N$ is
\begin{equation}
\label{eq:coefficientforrecurrence} \lambda_N = Nr(Nr+1) + n
c_0^{n-1} = \frac{4}{(n-1)^2} \left(N-\frac{n+1}{2}\right)(N+n).
\end{equation}
Clearly, the claim follows if $\lambda_N$ is different from $0$,
which is true since $n$ is even and positive.

Now we assume $n = 2m-1$, $m \ge 2$. The induction process used in
the previous case can performed provided that $\lambda_N \neq 0$.
This is true for $N \neq m$. Hence, the claim holds for $1 \le N \le
m-1$. Let $\tilde \phi_{m-1}(z) = c_0/z^r+ \dots+
c_{m-2}/z^{(m-1)r}$ be the corresponding function. It satisfies,
\begin{equation}
\label{eq:remainderphi0N-1} \epsilon_{m-1}(z) = \epsilon(\tilde
\phi_{m-1})(z) = \frac{A_{m-1}}{z^{mr+2}} + \gen_{(m+1)r+2}.
\end{equation}

Now we consider the case $N = m$. Since $\lambda_m = 0$, this case
cannot be dealt as before. We need to include logarithms in the
formal series.

Notice that, from~\eqref{eq:formalDelta2},
\begin{equation}
\label{eq:Delta2onlogs} \Delta^2\left(\frac{\log^{\ell}
z}{z^{kr}}\right)  = kr(kr+1) \frac{\log^{\ell} z}{z^{kr+2}}
-{\ell}(2kr+1) \frac{\log^{\ell-1} z}{z^{kr+2}} + {\ell}({\ell}-1)
\frac{\log^{\ell-2} z}{z^{kr+2}}  + \gen_{kr+4,\ell} .
\end{equation}
%\begin{align}
%\label{eq:Delta2onlogs} \Delta^2\left(\frac{\log^{\ell}
%z}{z^{kr}}\right)  = &kr(kr+1) \frac{\log^{\ell} z}{z^{kr+2}}
%-{\ell}(2kr+1) \frac{\log^{\ell-1} z}{z^{kr+2}} + {\ell}({\ell}-1)
%\frac{\log^{\ell-2} z}{z^{kr+2}}  \notag \\ &+ \gen_{kr+4,\ell} .
%\end{align}
We look for $\tilde \phi_m  =  \tilde \phi_{m-1} + \psi_m$
satisfying the claim, with $\psi_m(z) = c_{m-1,1}z^{-mr}\log z+
c_{m-1,0}z^{-mr}$. Hence we have that
$$
\epsilon_{m}   =   \epsilon_{m-1} + \Delta^2 (\psi_m) - \gg(\tilde
\phi_{m-1} + \psi_m)+ \gg(\tilde \phi_{m-1}).
$$
From~\eqref{eq:Delta2onlogs}, we have that
\begin{equation}
\label{eq:Delta2psim} \Delta^2 (\psi_m)(z) =
\frac{mr(mr+1)}{z^2}\psi_m(z) -(2mr+1) c_{m-1,1}\frac{1}{z^{mr+2}} +
\gen_{mr+4,1},
\end{equation}
while, from 2) in Lemma~\ref{lem:auxiliarylemma},
\begin{equation}
\label{eq:diffphimpsim}
\gg(\tilde \phi_{m-1}(z) + \psi_m(z))- \gg(\tilde \phi_{m-1}(z))
= -n\frac{c_0^{n-1}}{z^{2}}\psi_m(z)+\gen_{(m+1)r+2,L},
\end{equation}
with $L=[N/(m-1)]=[m/(m-1)]$.

Hence, substituting~\eqref{eq:Delta2psim}
and~\eqref{eq:diffphimpsim} into the expression for $\epsilon_m$
above, we obtain
$$
\epsilon_{m} (z) = \epsilon_{m-1} (z) +\frac{\lambda_m}{z^2}
\psi_m(z) -(2mr+1) c_{m-1,1}\frac{1}{z^{mr+2}} + \gen_{(m+1)r+2,L},
$$
where the coefficient $\lambda_N$ was introduced
in~\eqref{eq:coefficientforrecurrence} and, in fact, satisfies
$\lambda_m = 0$. Since $\epsilon_{m-1}(z) = A_{m-1} z^{mr+2}+
\gen_{(m+1)r+2}$ (see \eqref{eq:remainderphi0N-1}), taking
$c_{m-1,1} = A_{m-1}/(2mr+1)$, we have that $ \epsilon_m =
\gen_{(m+1)r+2,L}. $ Notice that the coefficient $c_{m-1,0}$ is
free.  Hence, the claim is proven for $1 \le N \le m$.

Now proceeding by induction the result is proven.
\end{proof}

\section{A solution of the inner equation}\label{sec:solinner}

The goal of this section is to prove the existence of a solution of the
%\emph{inner equation}
inner equation
satisfying
the properties stated in Theorem \ref{existencetheorem}.

For any $\g,\r>0$, we recall the complex domains
\begin{equation*}
D_{\g,\r}^{\s}=\{z\in \C : |\im z | > -\g \re z + \rho\},\;\;\;\; D_{\g,\r}^{\u} = -D_{\g,\r}^{\s}.
\end{equation*}
defined in \eqref{defDus} (see Figure \ref{dominiestable}).
We also introduce the norms
\begin{equation*}
\Vert \fB \Vert_{\nu,\g,\r}^{\u,\s} = \sup_{z \in D_{\g,\r}^{\u,\s}} |z^{\nu} \fB(z)|
\end{equation*}
and the Banach spaces
\begin{equation*}
\Ein^{\u,\s}_{\nu,\g,\r} = \{ \fB : D_{\g,\r}^{\u,\s} \to \C \;\;\text{such that} \;\; \Vert \fB \Vert_{\nu,\g,\r}^{\u,\s} <+\infty\}.
\end{equation*}

We also define the functional space
\begin{equation*}
\Ein^{\u,\s}_{\nu,k,\g,\r} = \{ \fB : D_{\g,\r}^{\u,\s} \to \C \;\;\text{such that} \;\; \overline{\fB}(z):=(\log z)^{-k} \fB(z)\in \Ein^{\u,\s}_{\nu,\g,\r} \}
\end{equation*}
and, if there is no danger of confusion, we will simply denote them
\begin{equation*}
\Ein_{\nu} = \Ein^{\u,\s}_{\nu,\g,\r},\;\;\;\; \Ein_{\nu,k}^{\log} = \Ein^{\u,\s,}_{\nu,k,\g,\r},\;\;\;\; \Vert \cdot \Vert_{\nu} = \Vert \cdot \Vert_{\nu,\g,\r}^{\u,\s},\;\;\;\;
D_{\g,\r} = D_{\g,\r}^{\u,\s}.
\end{equation*}

From now on we will denote by $C$ a generic positive constant independent of $\g,\r,\nu$. We state (without proof) the following
lemma which will be used without mention along this section.
\begin{lemma}\label{Banachspacelemma}
Let $0<\nu_1,\nu_2$. For any $\fB_1\in \Ein_{\nu_1}$ and $\fB_2\in
\Ein_{\nu_2}$, then
\begin{equation*}
\fB_1\cdot \fB_2 \in \Ein_{\nu_1+\nu_2} \;\;\;\;\;
\text{and}\;\;\;\;\; \Vert \fB_1 \cdot \fB_2 \Vert_{\nu_1 + \nu_2}
\leq \Vert \fB_2 \Vert_{\nu_2} \cdot \Vert \fB_1 \Vert_{\nu_1}.
\end{equation*}
Also there exists $C>0$ such that if $0<\nu_1 <\nu_2$  and $\fB\in
\Ein_{\nu_2}$, then
\begin{equation*}
\fB\in \Ein_{\nu_1} \;\;\;\;\; \text{and}\;\;\;\;\; \Vert \fB
\Vert_{\nu_1} \leq C\r^{-(\nu_2-\nu_1)} \Vert \fB \Vert_{\nu_2}.
\end{equation*}
\end{lemma}

As in previous section, we will denote by $\gen_{\nu}$ and $\gen_{\nu,k}$ a generic function belonging to $\Ein_{\nu}$ and $\Ein_{\nu,k}^{\log}$ respectively.

Theorem~\ref{existencetheorem} is rephrased in terms of the Banach
spaces $\Ein_{\nu,\g,\r}^{\u,\s}$ in the following proposition:
\begin{proposition}\label{prop:existence}
Given $\g>0$, there exists $\r_0>0$ such that for any $\r\geq \r_0$,
the
%\emph{inner equation}~\eqref{innereq}
inner equation~\eqref{innereq}
(polynomial case) and~\eqref{innereqt} (trigonometric case)
\begin{equation}\label{innereq2}
\Delta^2 (\phi) =  g(\phi)
\end{equation}
have exactly two solutions $\phi^{\u,\s}$ of the form
$$\phi^{\u, \s} = \fa + \f^{\u,\s}$$ where $\fa$ is the truncated
series of order $n$ defined in \eqref{def:phi0even},
\eqref{def:phi0oddp} and \eqref{def:phi0oddt}, depending on the case
we are dealing with, and $\f^{\u,\s} \in \Ein_{r+2,\g,\r}^{\u,\s}$.
\end{proposition}

The properties of $\phi_0$ we are interested in follow from
Proposition~\ref{prop:formalsolutionsoftheinnerequationconcreta}.
\begin{corollary}\label{cor:firstaproximation}
Let us consider the remainder of order $n$:
\begin{equation*}
\epsilon_0= \epsilon(\phi_0)=\Delta^2
(\phi_0) -g(\phi_0)
\end{equation*}
where $\phi_0$ is the truncated series of order $n$ defined
in~\eqref{def:phi0even}, \eqref{def:phi0oddp}
and~\eqref{def:phi0oddt}.

For any $\g>0$ there exists $\r_0$ big enough such that
\begin{enumerate}
\item If $n$ is even, $\phi_0= c_0 z^{-r} + \gen_{2r}$, in the polynomial case,
and $\phi_0 = \frac{r}{2} \log(-r z^{-2}) + \gen_r$, in the trigonometric one.
\item If $n=2m-1$ is odd, for the polynomial case $\phi_0= c_0 z^{-r} + \gen_{2r} + \gen_{mr,1}$.
Notice that, since $m\geq 2$, in particular we also have that
$\phi_0= c_0 z^{-r} +\gen_{2r,1}$. In the trigonometric case, we
have that $\phi_0 = \frac{r}{2} \log(-r z^{-2}) + \gen_{r} +
\gen_{(m-1)r,1}$ which also implies that $\phi_{0} = \frac{r}{2}
\log(-r z^{-2}) + \gen_{r,1}$.
\item For any value of $n$ we have that $\epsilon_0 \in \Ein_{nr+2}$.
\end{enumerate}
\end{corollary}

The proof of Proposition \ref{prop:existence} is performed in two steps. In
Section~\ref{subsec:linearizatedequation} we introduce a linear
equation which is close to the first order variational equation
of~\eqref{innereq2}  with respect to $\fa$. Such linear equation can
be easily inverted in the adequate Banach spaces. Finally, in
Section~\ref{subsec:fixedpointinner} we look for $\f^{\u,\s}$ as a solution
of a suitable fixed point equation.

From now on we will only deal with the $-\u-$ case, being the $-\s-$ case analogous. For that reason we will skip $-\u-$ from our notation.

\subsection{The linearized inner equation}\label{subsec:linearizatedequation}
We introduce the function
\begin{equation}\label{defH}
H(z) = (1+z^{-1})^{\eH}-2 + (1-z^{-1})^{\eH}
\end{equation}
for both cases, the polynomial and the trigonometric one with $\eH$ defined in~\eqref{defeH}.
In this section we are going to study the following linear
homogeneous second order difference equation:
\begin{equation}\label{linealequationH}
\Delta^2(\phi)(z) = H(z) \phi(z).
\end{equation}

We recall that the wronskian of two solutions, $\phi_1, \phi_2$ of a
linear difference equation is defined as:
\begin{equation*}
W(\phi_1,\phi_2)(z) = \left | \begin{array}{cc} \phi_1(z) & \phi_2(z) \\ \phi_1(z+1) & \phi_2(z+1) \end{array}\right  |.
\end{equation*}
In addiction, on the one hand, equation \eqref{linealequationH} has
the obvious solution $\vars(z) = z^{\eH}$ and, on the other hand, it
is a well known fact that $\varp = b\cdot \vars $ is a solution of
\eqref{linealequationH} if and only if
\begin{equation*}
\Delta b (z) = \frac{1}{\vars(z) \cdot \vars(z+1)} = \frac{1}{z^{\eH} (z+1)^{\eH}}.
\end{equation*}
One can also deduce that $W(\varp,\vars)\equiv 1$.

We will need a right inverse of the linear operator $\Delta$ defined
in appropriate Banach spaces. For this reason we introduce the
formal operator
\begin{equation}
\label{def:Delta-1} \Delta^{-1}(h)(z) = \sum_{k\geq 1} h(z-k)
\end{equation}
We emphasize that we are dealing with the unstable case.
\begin{lemma}\label{lemmaDelta1} Let $\alpha >0$. For any $\g>0$ there exists $\r_0>0$
such that, for any $\r\geq \r_0$, $\Delta^{-1}:\Ein_{\alpha+1,\g,\r}
\to \Ein_{\alpha,\g,\r}$ is a right inverse of the operator $\Delta$
defined in~\eqref{Delta1} with $\Vert \Delta^{-1} \Vert \leq C$.
\end{lemma}
The proof of this lemma is straightforward and can be found in \cite{Gel99}.

The first variational around $\fa$ of the inner
equation~\eqref{innereq2} is given by
\begin{equation}\label{firstvariationalinner}
\Delta^2(\phi)  = \Delta^2(\phi)  = D\gg(\fa) \phi
\end{equation}
and we notice that
\begin{equation*}
D\gg(\fa) = \left \{ \begin{array}{ll}
-n\fa^{n-1} + D\G(\fa) & \text{polynomial case}  \\
-(n-1) \ee^{\fa(n-1)} + D\G(\ee^{\fa})  \ee^{\fa} &
\text{trigonometric case.}
\end{array}\right .
\end{equation*}

By using the identities $c_0^{n-1}=-r(r+1)$ and $nr=r+2$, the fact
that $H(z)= (\eH-1) \eH z^{-2}+\gen_{3}$ and
Corollary~\ref{cor:firstaproximation}, the following result follows:
\begin{lemma}\label{lemmaleqH} For any $\g>0$ there exists $\r_0>0$ big enough such that
\begin{enumerate}
\item The function $H(z)$ satisfies
$ H = Dg(\fa) - \OL $, with $\OL \in \Ein_{r+2}$, if
$n\neq 3$, and $\OL \in \Ein_{r+2,1}^{\log}$, if $n=3$,
\item The function $\vars(z)=z^{\eH}$ is a solution of equation~\eqref{linealequationH}.
Consequently,
the function $\varp$ defined by
\begin{align*}
\varp(z) &= z^{\eH} \sum_{k>0} \frac{1}{(z-k)^{\eH}(z-(k+1))^{\eH}},
\end{align*}
is also an independent solution with $W(\varp,\vars)=1$. By
Lemma~\ref{lemmaDelta1}, $\varp\in \Ein_{\eH-1}$
\end{enumerate}
\end{lemma}

We notice that property~(1) of Lemma~\ref{lemmaleqH} implies that
the linear equation~\eqref{linealequationH} is a good approximation
of the first order variational with respect to
$\fa$~\eqref{firstvariationalinner}.

Finally, as we will see in the lemma below, Lemma~\ref{lemmaleqH}
allows us to invert the linear operator
$\lop(\phi)(z)=\Delta^2(\phi)(z) - H(z) \phi(z)$.
\begin{lemma}\label{inloplemma} For any $\g>0$, there exists $\r_0>0$ such that for any $\r\geq \r_0$,
the operator $\lop(\phi)=\Delta^2(\phi) - H\cdot  \phi$ has right
inverse $\inlop:\Ein_{\alpha+2,\g,\r} \to \Ein_{\alpha,\g,\r}$ i if
$\alpha>\eH-1$ and it has the expression
\begin{equation}\label{inlopdef}
\inlop( h ) = \varp \cdot \Delta^{-1}(\vars \cdot h) - \vars \cdot \Delta^{-1}(\varp \cdot h).
\end{equation}
Moreover, $\Vert \inlop (h) \Vert_{\alpha,\g,\r} \leq C \Vert h \Vert_{\alpha+2,\g,\r}$ being $C$ an independent constant of $\g,\r$.
\end{lemma}
\begin{proof}
We will skip $\g,\r$ from the notation. On the one hand
$\varp,\vars$ are independent solutions of the homogeneous linear
equation $\lop(\phi)=0$ and hence, by the variation of constants
method, we obtain formula~\eqref{inlopdef}. On the other hand, if
$g\in \Ein_{\alpha+2}$ with $\alpha>\eH-1$, then,  $\vars\cdot g \in
\Ein_{\alpha+2-\eH} $ and $\varp \cdot g \in \Ein_{\alpha+\eH+1}$ and by
Lemma~\ref{lemmaDelta1}, $\varp \cdot \Delta^{-1}(\vars \cdot g) \in
\Ein_{\alpha} $ and  $\vars \cdot \Delta^{-1}(\varp \cdot g)\in
\Ein_{\alpha}$. The bound $\Vert \inlop (g) \Vert_{\alpha} \leq C
\Vert g \Vert_{\alpha+2}$ is obtained by a direct application of
Lemma~\ref{lemmaDelta1}.
\end{proof}

\subsection{The fixed point equation}\label{subsec:fixedpointinner}
In this section we are going to prove Proposition
\ref{prop:existence} about the existence and properties of solutions
of the
%\emph{inner equation}~\eqref{innereq}
inner equation~\eqref{innereq}
(polynomial case)
and~\eqref{innereqt} (trigonometric case)
\begin{equation*}
\Delta^2 (\phi) = -\gg(\phi)
\end{equation*}
of the form $\phi = \fa + \f$, with $\fa$ given
by~\eqref{def:phi0even} ($n$ even), \eqref{def:phi0oddp} ($n$ odd,
polynomial case) or~\eqref{def:phi0oddt} ($n$ odd, trigonometric
case).

We introduce
\begin{equation}\label{restoeqpsi}
\R0  =  -\Delta^2 (\fa) +\gg(\fa),\;\;\;\;\;\;
 \mathcal{R}(\f) =  \f^2 \int_{0}^1 D^2 \gg(\fa + \lambda \f) (1-\lambda) \,d\lambda
\end{equation}
and we note that if $\phi=\fa + \f$ is a solution of the
%\emph{inner equation},
inner equation
then, by (1) of Lemma \ref{lemmaleqH} $\f$ has to satisfy
the second order difference equation given by
\begin{comment}
\begin{equation*}
\Delta^2 (\f)
%= -\gg(\fa+\f) - \Delta^2(\fa)
= H\cdot \f - \big [\Delta^2 (\fa)  - \gg(\fa)\big ] + \OL \cdot \f
+ \f^2 \int_{0}^1 D^2 \gg(\fa + \lambda \f) (1-\lambda) \,d\lambda
%=& H\f - \big [\Delta^2 (\fa) + \fa^n  - \G(\fa)\big ] + \OL \f +
%\f^2 \int_{0}^1 D^2 \gg(\fa + \lambda \f) (1-\lambda) \,d\lambda
\end{equation*}
%\marginpar{potser caldria decidir si es fa servir la notacio $F(\f)(z) $ o $F(\f(z))$.}
In order to shorten the notation, we introduce
\begin{equation}\label{restoeqpsi}
\R0  =  -\Delta^2 (\fa) - \fa^n + \G(\fa),\;\;\;\;\;\;
 \mathcal{R}(\f) =  \f^2 \int_{0}^1 D^2 \gg(\fa + \lambda \f) (1-\lambda) \,d\lambda
\end{equation}
and we rewrite the equation that $\f$ satisfies as:
\end{comment}
\begin{equation}\label{psiequation}
\Delta^2 (\f)  - H\cdot \f = \R0 +  \OL\cdot  \f + \mathcal{R}(\f)
\end{equation}
\begin{comment}
We now resume useful properties from $\R0$ and $\OL$. The results
follow directly from Corollary~\ref{cor:firstaproximation}
and Lemma~\ref{lemmaleqH}.
\begin{lemma}\label{RDproplemma} We have that, for any $\g>0$, there exists $\r_0$ such that
\begin{enumerate}
\item the function $\R0 \in \Ein_{r+4,\g,\r_0} $,
\item $\OL \in \Ein_{r+2,\g,\r_0}$ if $n\neq 3$ and $\OL \in \Ein_{r+2,1,\g,\r_0}$ if $n=3$.
\end{enumerate}
\end{lemma}
\end{comment}
As we proved in Lemma~\ref{inloplemma}, the linear operator~$\lop$
has a right-inverse in some adequate Banach spaces. Using it, we
will obtain a solution of the equation~\eqref{psiequation} by using
the fixed point equation given by
\begin{equation}\label{psifixedpointequation}
\f=\mathcal{F}(\f):=\inlop(\R0) + \inlop(\OL \cdot \f) + \inlop
\circ \mathcal{R}(\f).
\end{equation}
\begin{proposition} \label{existenceproposition}
Let $\g>0$. There exists $\r_1>0$ big enough such that, for any $\r\geq \r_1$, the fixed point
equation \eqref{psifixedpointequation} has a unique solution $\f\in
\Ein_{r+2,\g,\r}$.
\end{proposition}
\begin{proof}
We first note that there exists $\r_0>0$ such that $\inlop(\R0) \in \Ein_{r+2,\g,\r_0}$ since, by
Corollary~\ref{cor:firstaproximation}, $\R0\in \Ein_{r+4,\g,\r_0}$ if $\r_0$ is large enough.
Let $\rpf =  2\Vert \inlop(\R0) \Vert_{r+2,\g,\r_0}$.
Along the proof of this proposition we will denote by $K$ a generic
constant depending only on $\phi_0$, $\r_0$ and $\g$ and we will omit the dependence on $\g$ and $\r$
in the Banach spaces and norms.

Let $\f_1,\f_2 \in B(\rpf)\subset \Ein_{r+2}$. We start by bounding
the difference $\Vert \mathcal{F}(\f_1) -\mathcal{F}(\f_2)
\Vert_{r+2}$. By Lemma \ref{lemmaleqH} we have that, taking $\nu_{r} = 0$, if $n\neq 3$, and
$\nu_r = r/2$, if $n=3$, we have that $\OL\in \Ein_{r+2-\nu_r,\g,\r_1}$
provided that $\r_1$ is large enough. Henceforth, if $\f \in
\Ein_{r+2}$, $\OL \cdot \f \in \Ein_{2r+4-\nu_r}$. Applying
Lemmas~\ref{Banachspacelemma} and~\ref{inloplemma} we can easily
check that
\begin{equation}\label{boundLipD}
\Vert \inlop (\OL \cdot (\f_1 - \f_2)) \Vert_{r+2} \leq  C \r_1^{-r+\nu_r}
\Vert \OL\Vert_{r+2-\nu_r} \cdot \Vert \f_1 - \f_2 \Vert_{r+2}.
\end{equation}
Now we deal with $\mathcal{R}(\f_1) - \mathcal{R}(\f_2)$. We recall
that $\mathcal{R}$ was defined in~\eqref{restoeqpsi}. We notice
that
\begin{align}\label{defR1}
\mathcal{R}(\f_1) - \mathcal{R}(\f_2) = &
\big(\f_1^2 -\f_2^2 \big)  \int_{0}^{1} D^2 \gg(\fa+ \lambda \f_1)(1-\lambda) \,d\lambda \notag \\
& +\f_2^2  \int_{0}^{1} [D^2 \gg(\fa+ \lambda \f_1)-D^2
\gg(\fa+\lambda\f_2)](1-\lambda) \,d\lambda .
\end{align}
We first claim that, if $\lambda \in [0,1]$ and $z\in D_{\g,\r_1}$ with $\r_1$ big enough,
\begin{equation}\label{D2gR}
|D^{2}\gg(\fa(z)+ \lambda \f_1(z))|  \leq K  |z|^{\eH-4} \leq K |z|^{-2+r}
\end{equation}
where $\eH$ was defined in~\eqref{defeH}.
Indeed, we deal first with the polynomial case. In this case, by definition~\eqref{def:F} of $\gg$, there exists a constant
$K$ such that $|\gg(y)|\leq K |y|^{n}$. Moreover since $\gg$ is an analytic function, Cauchy's theorem
implies that, if $y_0 \in \mathbb{D}(\rF/2)$
\begin{equation}\label{Cauchyestimates}
|D^{2} \gg(y_0)| \leq K  |y_0|^{-2} \sup_{|y-y_0| \leq |y_0|/2} |\gg(y)| \leq K |y_0|^{n-2}.
\end{equation}
Also, since $\f_1 \in B(\rpf) \subset \Ein_{r+2}$, there exist
constants $0<K_1\leq  K_2$ and $\r_1$ big enough, $K_1 |z|^{-r} \leq
|\fa (z) + \lambda \f_1(z) | \leq  K_2 |z|^{-r} < \rF/2$ for any
$\lambda \in [0,1]$ and $z\in D_{\g,\r_1}$. Then, using $nr=r+2$ and
estimate~\eqref{Cauchyestimates}
\begin{equation*}
|D^{2}\gg(\fa(z)+ \lambda \f_1(z))|  \leq K|\phi_0(z) + \lambda \f_1(z) |^{n-2}
\leq  K  |z|^{2r} |z|^{-rn} = K  |z|^{-2+r}
\end{equation*}
which proves bound~\eqref{D2gR} in the polynomial case.  The trigonometric case is easier since
$|\gg(y)|\leq K |\ee^{y(n-1)}|$ and henceforth, a standard Cauchy estimate leads to bound~\eqref{D2gR}.
%\begin{equation*}
%\cor{|D^2 \gg ((\fa(z)+ \lambda \f_1(z))| \leq K |\ee^{\fa(z) + \lambda \f_1(z)}| \leq K |z|^{-2}.}
%\end{equation*}
Hence, if $\f_1,\f_2 \in B(\rpf)$,
\begin{align*}
|D^2 \gg(\fa(z)+ \lambda (\f_1(z)))\cdot (\f_1^2(z) -\f_2^2(z))|\leq K |z|^{-4} |\f_1(z) -\f_2(z)|.
\end{align*}

Now we claim that, for $\lambda \in [0,1]$ and $\f_1,\f_2\in
B(\rpf)$,
$$
|[D^2 \gg(\phi_0(z)+ \lambda (\f_1(z)))- D^2
\gg(\phi_0(z)+\lambda(\f_2(z)))] \f_2^2(z)| \leq   K|z|^{-2r-4}
|\f_1(z) - \f_2(z)|.
$$
%\begin{align*}
%|[D^2 \gg(\phi_0(z)+ \lambda (\f_1(z)))- &D^2
%\gg(\phi_0(z)+\lambda(\f_2(z)))] \f_2^2(z)| \leq  \\ & K|z|^{-2r-4}
%|\f_1(z) - \f_2(z)|.
%\end{align*}
Indeed, since $\gg$ is an analytic function, $D^3 \gg$ is bounded in
$\mathbb{D}(\rF)$ and henceforth,  for any $y_1,y_2\in
\mathbb{D}(\rF)$,
$
|D^2 \gg(y_1) - D^2 \gg(y_2) |\leq K|y_1-y_2|
$
and the claim is proved provided that $\r_1$ is large enough to
ensure that for any $z\in D_{\g,\r_1}$, $\f_1(z),\f_2(z) \in
\mathbb{D}(\rF)$.

Finally by using the previous computations and
formula~\eqref{defR1}, one obtains that $\mathcal{R}(\f_1)
-\mathcal{R}(\f_2) \in \Ein_{r+6} \cup \Ein_{(2r+4) + r+2} =
\Ein_{r+6} \subset \Ein_{r+4}$ and moreover
\begin{equation}\label{boundLipR1}
\Vert \mathcal{R} (\f_1) -  \mathcal{R} (\f_2)\Vert_{r+4} \leq C |\r|^{-2} \Vert \f_1 - \f_2\Vert_{r+2}.
\end{equation}
Then, by Lemma~\ref{inloplemma}, $\inlop(\mathcal{R}(\f_1) -
\mathcal{R}(\f_2)) \in \Ein_{r+2}$ and moreover
\begin{equation*}
\Vert \inlop(\mathcal{R}(\f_1) - \mathcal{R}(\f_2))\Vert_{r+2} \leq  C |\r|^{-2} \Vert \f_1 - \f_2 \Vert_{r+2}.
\end{equation*}
Using this bound, \eqref{boundLipD} and
definition~\eqref{psifixedpointequation} of the
operator~$\mathcal{F}$, one has that, if $\r_1$ is large enough and
$\r\geq \r_1$
\begin{equation*}
\Vert \mathcal{F}(\f_1) - \mathcal{F}(\f_2) \Vert_{r+2} \leq C \r_1^{-r+\nu_r} \Vert \f_1 - \f_2\Vert_{r+2} \leq \frac{1}{2} \Vert \f_1 - \f_2\Vert_{r+2}
\end{equation*}
and hence $\mathcal{F}$ is contractive (we recall that $r-\nu_r>0$). Moreover, if $\f\in B(\rpf)$,
\begin{equation*}
\Vert \mathcal{F} (\f) \Vert_{r+2} \leq \Vert \mathcal{F}(0) \Vert_{r+2} + \Vert \mathcal{F}(0)-\mathcal{F}(\f)\Vert_{r+2} \leq
\Vert \R0 \Vert_{r+2} + \frac{1}{2} \Vert \f \Vert_{r+2} < \rpf
\end{equation*}
which ends the proof of the Proposition.
\end{proof}

\section{The difference $\phi^{\u}-\phi^{\s}$}
\label{sec:difference} By Proposition \ref{prop:existence} the
existence of two solutions $\phi^{\u,\s} = \fa + \f^{\u,\s}$ of the
%\emph{inner equation}
inner equation
is proved.
Let us write $\difin = \phi^{\u}-\phi^{\s}$ and we also introduce the
function
\begin{equation}\label{defBinner}
\UR = -\int_{0}^{1} (1-\lambda) D^2 g(\phi^{\s} +
\lambda(\phi^{\u}- \phi^{\s}))\,d\lambda \cdot (\phi^{\u}-\phi^{\s} ).
\end{equation}

We recall that both
$\phi^{\u,\s}$ are solutions of the same nonlinear difference
equation:
\begin{equation}\label{psiequationinner}
\Delta^2 (\phi) = -\phi^{n} + \G(\phi)=-\gg(\phi).
\end{equation}
Consequently, the function $\difin$ satisfies the linear difference
equation
\begin{equation}\label{eqdif}
\Delta^{2} (\difin) = (-D\gg(\phi^{\s}) + \UR)\cdot \difin .
\end{equation}
Although we do not have a good representation of the difference
$\difin = \phi^{\u}-\phi^{\s}$, by means of Proposition
\ref{prop:existence} we already know it is well defined and some not
optimal bounds for~$\difin$ which allow us to define a new linear
equation from which $\difin$ is also a solution. In conclusion, we
will use $\difin = \phi^{u}-\phi^{s}$ both as a known function (to
define~$\UR(z)$) and as an unknown solution of the above linear
equation.

The goal of this section is to prove that any analytic solution of
equation~\eqref{eqdif} satisfying adequate boundary condition has to
be exponentially small, that is of~$\gen(\ee^{-2\pi \ic z})$. In
fact, as claimed in Theorem~\ref{differencetheorem}, we will provide
an exact formula for~$\difin$.
\subsection{Notation}
Given $\r,\g>0$, let us recall the complex domain
\begin{equation*}
E_{\g,\r}=D_{\g,\r}^{\u} \cap D_{\g,\r}^{\s} \cap \{ z\in \C:  \im z <0\}\backslash \{z\in \C : |\re z|\leq 1 , \; |\im z| \leq \r+\g\}
\end{equation*}
defined in~\eqref{defEdomain} (see Figure \ref{dominiiner}).

For $\nu,k\in \mathbb{R}$, we also introduce the norms
\begin{equation*}
\Vert \fB \Vert_{\nu,\g,\r} = \sup_{z \in E_{\g,\r} }|z^{\nu}  \fB(z)|,\;\;\;\;\;\;\;
\Vert \fB \Vert_{\nu,k,\g,\r}^{\log} = \sup_{z \in E_{\g,\r} }|z^{\nu} (\log z)^{-k} \fB(z)|,
\end{equation*}
and the Banach spaces
\begin{align*}
\EInt_{\nu,\g,\r} &= \{\fB : E_{\g,\r} \to \C \;\;\text{such that} \;\; \Vert \fB\Vert_{\nu,\g,\r} <+\infty\}, \\
\EInt_{\nu,k,\g,\r}^{\log} &= \{\fB : E_{\g,\r} \to \C \;\;\text{such that} \;\; \Vert \fB\Vert_{\nu,k,\g,\r}^{\log} <+\infty\},
\end{align*}
If there is no danger of confusion we will simply denote
\begin{equation*}
\EInt_{\nu} = \EInt_{\nu,\g,\r},\;\;\;\; \Vert \cdot \Vert_{\nu} = \Vert \cdot \Vert_{\nu,\g,\r},\;\;\;\;
\EInt_{\nu,k}^{\log} = \EInt_{\nu,k,\g,\r}^{\log},\;\;\;\; \Vert \cdot \Vert_{\nu,k}^{\log} = \Vert \cdot \Vert_{\nu,k,\g,\r}^{\log}.
\end{equation*}
\begin{lemma}\label{Banachspacelemmainner}
Let $0<\nu_1,\nu_2$. For any $f\in \EInt_{\nu_1}$ and $g\in
\EInt_{\nu_2}$, then $f\cdot g \in \EInt_{\nu_1+\nu_2}$ and
\begin{equation*}
\Vert f \cdot g \Vert_{\nu_1 + \nu_2} \leq \Vert f \Vert_{\nu_1} \cdot \Vert g \Vert_{\nu_2}.
\end{equation*}
Also, there exists a constant $C$  such that if $0<\nu_1 <\nu_2$ and
$f\in \EInt_{\nu_2}$
\begin{equation*}
f\in \EInt_{\nu_1} \;\;\;\;\; \text{and}\;\;\;\;\; \Vert f
\Vert_{\nu_1} \leq C\r^{-(\nu_2-\nu_1)} \Vert f \Vert_{\nu_2}.
\end{equation*}
\end{lemma}

As in previous sections, we will denote by $\gen_{\nu}$ and $\gen_{\nu,k}$ a generic function belonging to $\EInt_{\nu,\g,\r}$ and $\EInt_{\nu,k,\g,\r}^{\log}$
respectively.

\subsection{A right inverse of the operator $\Delta (\phi)(z) = \phi(z+1)-\phi(z)$}

In this section we are going to construct a right inverse of the linear operator $\Delta$:
\begin{equation}\label{delta1in}
\Delta (\phi)(z) = \phi(z+1) - \phi(z),
\end{equation}
defined on functions belonging to $\EInt_{\nu,k}^{\log}$ with $\nu,k \in \mathbb{R}$.
We will follow the results introduced in~\cite{FS01} (which provide
an explicit formula for $\Delta^{-1}$) and we also give useful properties
of this operator when it acts on $\EInt_{\nu+1,k}^{\log}$.

We first notice that, since $E_{\g,\r}$ is an open set, for any
$z\in E_{\g,\r}$ there exists $\sigma(z)$ such that $\{w\in \C :
|z-w|<2\sigma(z)\} \subset E_{\g,\r}$. In consequence, the complex
path $\gamma_z = \gamma_z^1 \vee \gamma_z^2$ (see
Figure~\ref{camiinner})
\begin{align}\label{path}
\gamma_z^1(t) &= \{ -i(\r+\g) (1-t) + t (z-\sigma(z)) , \;\; t \in [0,1)\}
,\notag \\ \gamma_z^{2}(t) &= \{ z -\sigma(z)+ it,\;\;\; t \in
(-\infty,0]\}
\end{align}
is contained into the complex set $E_{\g,\r}$.

\begin{figure}[h]
\begin{center}
  \includegraphics[width=10cm]{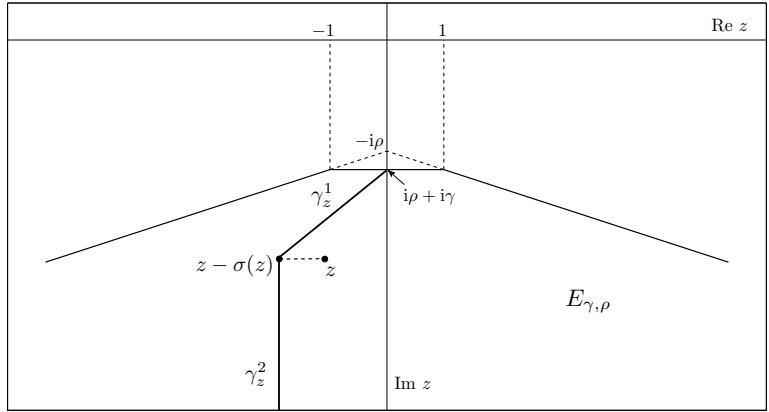}\\
  \caption{Path $\gamma_z$}\label{camiinner}
  \end{center}
\end{figure}

Given $\iff $ an analytic function and $z\in E_{\g,\r}$, we introduce the linear operators
\begin{equation}\label{deffinnunegpos}
\Delta^{-1}_{-}(h)(z) = \int_{\gamma_{z}} \frac{\iff(u)}{\ee^{2\pi
\ic (u-z)}-1} \,du \quad \text{ and } \quad \Delta^{-1}_{+}(h)(z) =
\int_{\gamma_{z}} \frac{\iff(u)}{1-\ee^{-2\pi \ic (u-z)}} \,du.
\end{equation}

\begin{proposition}\label{invDeltainner}
Let $\nu,k \in \mathbb{R}$ and $\gamma>0$. We define the linear operator
\begin{equation*}
\Delta^{-1} = \left \{ \begin{array}{ll} \Delta^{-1}_- & \text{if}\;\;\;\; \nu\leq 0 \\ \Delta^{-1}_+  & \text{if}\;\;\;\; \nu>0. \end{array}\right.
\end{equation*}
There exists $\r_0>0$ such that, for any $\r\geq \r_0$,
\begin{enumerate}
\item if $\nu\neq 0$, $\Delta^{-1}:\EInt^{\log}_{\nu+1,k,\g,\r} \to \EInt^{\log}_{\nu,k,\g,\r}$ is a right inverse of the
operator~$\Delta$.
\item if $\nu= 0$, $\Delta^{-1}:\EInt^{\log}_{1,k,\g,\r} \to \EInt^{\log}_{0,k+1,\g,\r}$ is a right inverse of the
operator~$\Delta$.
\end{enumerate}
Moreover in both cases, there exists a positive constant $C$ such that
$
\Vert \Delta^{-1} \Vert \leq C.
$
\end{proposition}

\begin{proof}
Along this proof we will denote by $K$ a generic constant depending
only on~$\g$ and~$\nu$. We will skip $\g,\r_0$ and $\r$ from our
notation of the Banach spaces and norms.

We fix $\nu\in \mathbb{R}$, $\g>0$ fulfilling the hypotheses of Proposition \ref{invDeltainner} and $\r_0\leq \r$ big enough. Let $\iff\in \EInt_{\nu+1,k}^{\log}$,
and we introduce $\fin=\Delta^{-1}(h)$.
Our first observation is that $\fin$ is an analytic function defined in $E_{\g,\r}$. Indeed, for any $\sigma_0>0$ we define the set
\begin{equation*}
\Omega_{\sigma_0}=\{ u \in E_{\g,\r} : u-\sigma_0 \in E_{\g,\r}\}.
\end{equation*}
We emphasize that $E_{\g,\r} = \cup_{\sigma_0>0} \Omega_{\sigma_0}$.
Moreover we note that, if $z\in \Omega_{\sigma_0}$ we can take
$\sigma(z) = \sigma_0$ in the expression~\eqref{deffinnunegpos}
of~$\fin(z)$. Henceforth, in order to deduce that $\fin$ is an
analytic function in $\Omega_{\sigma_0}$, we only have to study the
convergence of
\begin{equation*}
\int_{\gamma_{z}^2} \frac{\iff(u)}{\ee^{\mp 2\pi \ic(u-z)}-1}\, du
\end{equation*}
To this end, we observe that $ |\ee^{\mp 2\pi \ic (\gamma^{2}_z -z)}
-1| \geq \left | \ee^{\pm t 2\pi |\r + \g+\im z|} -1|\right . $ and
that $ |h(\gamma_z^2(t))|\leq C(z)  |t|^{-\nu-1}\log^k (|t|) $ for
some function $C(z)$. Therefore, if $\nu \leq 0$ and $t\in
(-\infty,0]$,
\begin{equation*}
\frac{|h(\gamma_z^2(t))|}{|\ee^{2\pi \ic (\gamma^{2}_z -z)} -1|}  \leq 2 C(z)|t|^{-\nu-1}\log^k (|t|) \ee^{t 2\pi |\r + \im z|}
\end{equation*}
and we are done for the case $\nu \leq 0$. The case $\nu>0$ can be done analogously.

Now we are going to check that $\Delta^{-1}$ is a right inverse of the operator $\Delta$. We take into account that, if $z,z+1\in E_{\g,\r}$
\begin{equation*}
\Delta \fin (z) = \mp \int_{\g_{z+1} - \g_{z}} \frac{\iff(u)}{\ee^{\mp 2\pi \ic(u-z)}-1} \,du
\end{equation*}
and therefore, since the only singularity of $\mp \iff(u)/(\ee^{\mp 2\pi \ic(u-z)}-1)$ is $u=z$ and it is a simple pole with residue $\iff(z)/2\pi \ic$,
we have that both $\fin_{\pm}$ are solution of $\Delta(\fin)=h$ defined in the complex domain $E_{\g,\r}$. Here we have proceed exactly as in \cite{FS01}.

It only remains to prove that $\fin=\Delta^{-1}(h)\in  \EInt_{\nu,k}^{\log}$ provided $h\in  \EInt_{\nu+1,k}^{\log}$.
We restrict ourselves to the complex domain $\widetilde{E}_{\g,\r+\g}\subset E_{\g,\r}$ defined by:
\begin{equation*}
\widetilde{E}_{\g,\r'} = D_{\g,\r'}^{\u} \cap D_{\g,\r'}^{\s} \cap \{ z\in \C:  \im z <0\}.
\end{equation*}
We notice that, if the following bounds are proved,
\begin{align}
|z^{\nu} (\log z)^{-k} \fin(z)| &\leq K \Vert h \Vert^{\log}_{\nu+1,k}\;\;\;\; z\in \widetilde{E}_{\g,\r+\g} & \text{if}\;\;\;\; \nu\neq 0 \label{boundfindeltainvnu} \\
|(\log z)^{-k-1} \fin(z)| &\leq K \Vert h \Vert^{\log}_{1,k} \;\;\;\; z\in \widetilde{E}_{\g,\r+\g} & \text{if}\;\;\;\; \nu = 0\label{boundfindeltainvnu0}
\end{align}
the same statement holds for $z\in E_{\g,\r}$. Indeed, assume that bounds \eqref{boundfindeltainvnu} and \eqref{boundfindeltainvnu0} are satisfied and let
$z\in E_{\g,\r} \backslash \widetilde{E}_{\g,\r,\g}$. We have two cases, $\re z\leq 0$ and $\re z >0$.
On the one hand, if $\re \leq 0$, it is clear that $z+1\in \widetilde{E}_{\g,\r+\g}$ and that
$
\fin(z) = \fin(z+1)- h(z)
$. On the other hand, if $\re z>0$, $z-1\in \widetilde{E}_{\g,\r+\g}$ and consequently $\fin(z) = \fin(z-1) + h(z)$.
In any case, $|\fin(z)|\leq |\fin(z\pm 1)| + |h(z)|$. Here we have used that $\Delta(\fin)=h$.
Therefore, if $\nu \neq 0$, using bound \eqref{boundfindeltainvnu}, we obtain
\begin{equation*}
|\fin(z)| \leq K \Vert h \Vert_{\nu,k}^{\log} \big (|z\pm 1|^{-\nu} |\log (z\pm 1)|^{k}\big ) + |z|^{-\nu-1} |\log z|^{k} \big)
\leq K \Vert h \Vert_{\nu,k}^{\log}|z|^{-\nu} |\log z|^{k}
\end{equation*}
and the result is proved for $\nu \neq 0$.  Analogously we check the result for $\nu =0$.

The proof of bounds \eqref{boundfindeltainvnu} and \eqref{boundfindeltainvnu0} is easy but it requires tedious computations
which will be omit here. Nevertheless, we point out that for any fixed $z\in \widetilde{E}_{\g,\r+\g}$ we can take $\sigma(z)=1/2$
in the definition \eqref{path} of $\gamma_z$.
\end{proof}

\subsection{Two independent solutions of the linear equation \eqref{eqdif}}

We recall that $\difin = \phi^{\u} - \phi^{\s}$ satisfies equation \eqref{eqdif}:
\begin{equation}\label{eqdifinner}
\Delta^{2} \difin  = (-D\gg(\phi^{\s}) + \UR) \difin.
\end{equation}
The following lemma states the properties of $\UR$ we will need. Its proof is completely
analogous to the one of bound~\eqref{D2gR} provided $\phi^{\u} - \phi^{\s} \in \EInt_{r+2,\g,\r}$.
\begin{lemma} \label{Edif}Let $\g$ and $\r$ satisfying the conclusions of
Proposition~\ref{prop:existence} and $\UR$ be the function defined in \eqref{defBinner}.
We have that $\UR \in \EInt_{r+6-\eH,\g,\r}$.
\end{lemma}
\begin{comment}
\begin{proof}
On the one hand, if $y$ belongs to the analyticity domain of $\gg$,
$|D^{2}\gg(y)|\leq K |y|^{n-2}$ for some constant $K>0$ and, on the
other hand, from Proposition~\ref{prop:existence}, we have that
$|\phi^{u}(z)|\leq Kz^{-r}$ and $|\phi^{u}(z) - \phi^{s}(z) |\leq K
|z|^{r+2}$.
\end{proof}
\end{comment}

As we did in Section~\ref{subsec:linearizatedequation}, we split
\begin{equation*}
-D\gg(\phi^{s}) +\UR= H + \IL
\end{equation*}
where $H$ was defined in~\eqref{linealequationH}, $\IL \in
\EInt_{r+2}$, if $n\neq 3$, and $ \IL\in \EInt_{r+2,1}^{\log}$, if
$n=3$. We rewrite equation~\eqref{eqdifinner} as
\begin{equation*}
\Delta^2 (\difin) -H \cdot \difin = \IL\cdot \difin.
\end{equation*}

A solution of the homogeneous equation $ \Delta^2 (\fin)  =H\cdot
\fin$ is $\vars(z) = z^{\eH}$. The function
\begin{equation*}
\varp(z) = z^{\eH} \Delta^{-1}\left (\frac{1}{\vars(z+1) \vars(z)}\right ) \in \EInt_{\eH-1}
\end{equation*}
is another solution satisfying~$W(\varp,\vars)=1$.

By using these decomposition as well as
Proposition~\ref{invDeltainner} for the operator~$\Delta^{-1}$, we
can obtain solutions of the non homogeneous linear equation
$\Delta^2(\fin) -H\cdot \fin = h$.
\begin{lemma}\label{invlopinlemma} For any $\g>0$ there exists $\r_0>0$
large enough such that for any $\r\geq \r_0$, the operator
$\lopin(\fin) = \Delta^{2}(\fin) - H\cdot \fin$ has right inverse
defined  in $E_{\g,\r}$:
\begin{equation}
\inlopin(h)= \varp \cdot \Delta^{-1} (\vars \cdot h) - \vars \cdot \Delta^{-1} (\varp \cdot h).
\end{equation}
There exists $C>0$ such that for any $\alpha \in \mathbb{R}$ and
$h\in \EInt_{\alpha+2,\g,\r}$, we have:
\begin{enumerate}
\item If $\alpha \neq \eH -1$ and $\alpha \neq -\eH$ then $\inlopin(h) \in \EInt_{\alpha,\g,\r}$ and
$\Vert \inlopin(h) \Vert_{\alpha,\g,\r} \leq C \Vert h \Vert
_{\alpha+2,\g,\r}$.
\item If either $\alpha=\eH -1$ or $\alpha=-\eH$, $\inlopin(h) \in \EInt_{\alpha,1,\g,\r}^{\log}$ and
$\Vert \inlopin(h) \Vert_{\alpha,1,\g,\r}^{\log} \leq C \Vert h\Vert _{\alpha+2,\g,\r}$.
\end{enumerate}
\end{lemma}
%We skip the proof of this lemma since is quite analogous as the one
%of Lemma~\ref{inloplemma}.

Next lemma provides a fundamental system of solutions of the linear
equation~\eqref{eqdifinner}.
\begin{lemma} \label{sollinearequationinner}
Let $\g>0$. There exists $\r_0$ large enough such that for any
$\r\geq \r_0$, the equation~\eqref{eqdifinner} has two independent
solutions, $\varpin$ and~$\varsin$, satisfying
\begin{align*}
\varpin(z)  &= \partial_z \phi^{\s}(z) + \varpin^1(z),\;\;\;\;\; \varpin^{1} \in \EInt_{r+3,\g,\r}, \\
\varsin(z)&= \frac{z^{\eH}}{r d_{\eH}(2\eH-1)} + \varsin^1(z),\;\;\;\;\; \varsin^1 \in \EInt^{\log}_{\nu,k,\g,\r}
\end{align*}
with $\nu=\min\{r-\eH,1-\eH\}$, $k=0$ if $n\neq 3$, $k=1$ if $n=3$ and  $d_{\eH}$ is defined in~\eqref{defeH}, .
\end{lemma}
\begin{proof}
First we look for $\varpin$. By construction, $\partial_z \phi^{\s}$
is a solution of the variational equation
$
\Delta^2 \fin = -D\gg(\phi^s) \fin,
$
therefore, the equation that $\varpin^1$ has to satisfy is
\begin{equation}\label{varinner1}
\Delta^2 (\fin)  - H\cdot \fin  =  \IL\cdot \fin + \UR\cdot
\partial_z \phi^{\s}.
\end{equation}
We look for $\varpin^1$ by means of the fixed point equation
\begin{equation}\label{varinnerfixedpoint}
\fin  = \inlopin (\UR \cdot \partial_z \phi^{\s}) + \inlopin (\IL \cdot \fin).
\end{equation}
We are interested in solutions belonging to $\EInt_{r+3}$. It is
enough to check that the norm of the linear operator $\mathcal{G}
:\EInt_{r+3} \to \EInt_{r+3}$ defined by $\mathcal{G}(\fin) =
\inlopin (\IL \cdot \fin)$ is less than one. This fact follows from
Lemma~\ref{invlopinlemma} together with the fact that $\UR\in
\EInt_{r+6-\eH}$, $\IL \in \EInt_{r+2}$ if $n\neq 3$ and $\IL \in
\EInt_{r+2,1}^{\log}$ if $n=3$.  One easily then deduces that
\begin{equation}\label{expresionvarpin1}
\varpin^1 = (\text{Id} - \mathcal{G})^{-1} \big (\inlopin(\UR \cdot \partial_z \phi^{\s})\big )\in \mathcal{Y}_{r+3}
\end{equation}
is a solution of equation \eqref{varinner1}.

Now we deal with the second solution of the
equation~\eqref{eqdifinner}. We observe that, since~$\varpin$ is a
solution, then the function~$\varsin = \bb \cdot \varpin$ is also a
solution of the linear equation~\eqref{eqdifinner} satisfying
$W(\varpin,\varsin)=1$, if and only if $\bb$ satisfies
\begin{equation*}
\Delta (\bb)(z) = \frac{1}{\varpin(z+1) \varpin(z)}.
\end{equation*}

By Proposition \ref{prop:existence}, $\varpin= \partial_z \phi^{\s}+
\varpin^1= \partial_z \fa+ \partial_z\f^{\s}+\varpin^1 $ where
$\partial_z \f^{\s}, \varpin^1 \in \EInt_{r+3}$. Moreover, using the
definitions of $d_{\eH},\eH$ in~\eqref{defeH} and
Corollary~\ref{cor:firstaproximation}, it is a direct computation to
check that $\bb$ has to satisfy the linear equation
\begin{equation*}
\Delta(\bb) (z) = \frac{z^{2\eH-2}}{r^2 d_{\eH}^2}  + S(z)
\end{equation*}
with $S\in \EInt_{-2\eH+3}$ if $n=2$, $S\in \EInt_{-2\eH+3,1}^{\log}$ if $n=3$ and $S\in \EInt_{-2\eH+2+r}$ if $n>3$.
We take
\begin{equation*}
\bb_0(z)=\frac{z^{2\eH-1}}{r^2 d_{\eH}^2 (2\eH-1)}
\end{equation*}
and we note that $r^2 d_{\eH}^2 \Delta (\bb_0)(z) = z^{2\eH-2} + \gen_{-2\eH+3}$. Henceforth, the difference $\bb_1 = \bb - \bb_0$ satisfies
an equation of the form
\begin{equation}\label{eqc1innern<3}
\Delta(\bb_1) = \tilde{S}(z)
\end{equation}
with $\tilde{S}\in \EInt_{-2\eH+3}$ if $n=2$, $\tilde{S}\in \EInt_{-2\eH+3,1}^{\log}$ if $n=3$ and $\tilde{S}\in \EInt_{-2\eH+2+r}$ if $n>3$.
Applying Proposition \eqref{invDeltainner} one has that equation \eqref{eqc1innern<3} has a solution $\bb_1$ belonging to
$\EInt_{-2\eH+2}$ if $n=2$, $\EInt_{-2\eH+2,1}^{\log}$ if $n=3$ and $\EInt_{-2\eH+1+r}$ if $n>3$ and the result follows.
\end{proof}

\subsection{A final formula for $\difin = \phi^{\u} - \phi^{\s}$}
Since $\difin$ is a solution of the linear homogenous difference equation \eqref{eqdifinner},
the general
theory allows us to write it as
\begin{equation}\label{formdifin1}
\difin (z) = p_1(z) \eta_1(z) + p_2(z) \eta_2(z)
\end{equation}
with $\eta_1,\eta_2$ two independent solutions of \eqref{eqdifinner}
and $p_1, p_2$ $1$-periodic, analytic functions in $E_{\g,\r}$.
Moreover, if $W(\eta_1,\eta_2)=1$, the functions $p_1$ and $p_2$ are
determined by
\begin{equation}\label{formulap}
p_1(z) = W(\difin,\eta_2)(z), \quad p_2(z) = -W(\difin,\eta_1)(z).
\end{equation}

\begin{lemma} \label{diflemma}
Let $\g,\r>0$ and $\eta_1,\eta_2$ two independent solutions of the linear difference equation \eqref{eqdifinner}
satisfying that $W(\eta_1,\eta_2)=1$ and that $\eta_1 \in \EInt_{r+1}$ and $\eta_2 \in \EInt_{-\eH}$.

Then there exist coefficients $p_1^k,p_2^k$ (depending on $\eta_{1,2}$) such that
\begin{equation}\label{expresiondifin}
\difin(z)= \eta_1(z) \sum_{k<0} p_1^{k} \ee^{2\pi \ic k z} + \eta_2(z)\sum_{k<0} p_2^{k} \ee^{2\pi \ic k z}.
\end{equation}
\end{lemma}
\begin{proof}
We first point out that, we already know that $\difin =\phi^{\u}-
\phi^{\s} = \f^{\u}- \f^{\s} \in \EInt_{r+2,\g,\r}$ provided that,
by Theorem~\ref{existencetheorem}, $\f^{\u,\s} \in
\Ein^{\u,\s}_{r+2}$. In addition, if $h \in \EInt_{\nu,\g,\r}$, then
$\Delta (h) \in \EInt_{\nu+1,2\g,2\r}$. Indeed, standard arguments
can be used to prove that, if  $h\in \EInt_{\nu,\g,\r}$ then
$\partial_z h \in \EInt_{\nu+1,2\g,2\r}$ (see for
instance~\cite{Bal06}). Therefore, if $z,z+1\in E_{2\g,2\r}$,
\begin{equation*}
|h(z+1) - h(z)| \leq \Vert \partial_z h\Vert_{\nu+1} \int_{0}^1
\frac{1}{|z+t|^{\nu+1}} \leq K \Vert \partial_z h \Vert_{\nu+1} \frac{1}{|z|^{\nu+1}}.
\end{equation*}
Using the above property, that $\difin \in \EInt_{r+2}$ and formula
\eqref{formulap} for $p_1,p_2$, one has that $p_1 \in \EInt_{1}$ and
$p_2 \in \EInt_{r+4}$. In particular, $p_1,p_2 \to 0$ as $\im z \to
-\infty$ and since they are $1$-periodic:
\begin{equation*}
p_1(z) = \sum_{k<0} p_1^{k} \ee^{2\pi \ic k z},\;\;\;\;\;p_2(z) =
\sum_{k<0} p_2^{k} \ee^{2\pi \ic k z}
\end{equation*}
ant the lemma is proved.
\end{proof}

We recall that the existence of independent solutions of the linear
difference equation~\eqref{eqdifinner} satisfying the hypotheses of
Lemma~\ref{diflemma} is guaranteed by
Lemma~\ref{sollinearequationinner}. Henceforth Lemma~\ref{diflemma}
applied to $\varpin, \varsin$ already gives an expression of
$\difin$ which is exponentially small. Among other things, we have
proved that there exist $\g,\r>0$ such that
\begin{equation}\label{bounddifexp1}
|\ee^{2\pi \ic z} z^{-\eH} (\phi^{\u}(z)-\phi^{\s}(z))| \leq K\;\;\;\;\;\;\; z\in E_{\g,\r}.
\end{equation}

Nevertheless we have not proved
Theorem \ref{differencetheorem} yet. We need to look for more suitable independent solutions of \eqref{eqdifinner}
to apply Lemma \ref{diflemma}.

\begin{corollary}\label{corollaryleinner}
Let $\g>0$. There exists $\r_0>0$ big enough such that for any $\r\geq \r_0$, equation \eqref{eqdifinner} has two fundamental solutions of the form
\begin{align*}
\varpine(z)  &= \partial_z \phi^{\s}(z) + \varpine^1(z), \\
\varsine(z) &= \frac{z^{\eH}}{r d_{\eH}(2\eH-1)}(z) + \varsine^1(z),
\end{align*}
with with $d_{\eH}$ defined in~\eqref{defeH} and $\varpine^1$ satisfying
\begin{equation*}
\sup_{z\in E_{\g,\r}} | z^{1-\eH} \ee^{2\pi \ic z} \varpine^1(z)|<+\infty.
\end{equation*}
In addition, $\varsine^{1} \in \EInt_{r-\eH,\g,\r}$ if $n>3$, $\varsine^1 \in \EInt_{1-\eH,\g,\r}$ if $n=2$
 and $\varsine^1 \in \EInt_{1-\eH,1,\g,\r}^{\log}$ if $n=3$.
\end{corollary}
\begin{proof}
As in the proof of Lemma \ref{sollinearequationinner} we write $\varpine^1 = \varpine-\partial_z\phi^{\s}$. We note that $\varpine^1$ satisfies
equation \eqref{varinner1}:
\begin{equation*}
\Delta^2 (\varpine^1) (z) - H(z) \varpine^1 (z) =  \IL(z) \varpine^1(z) + \UR(z) \partial_z \phi^{\s}(z).
\end{equation*}
We write $\spli (z)= \ee^{2\pi \ic z} \varpine^1(z)$ and we notice that $\spli$ has to satisfy the equation
\begin{equation}\label{eqeta}
\Delta^2 (\spli)(z) -H(z) \spli(z) = \IL(z) \spli(z) +  \ee^{2\pi
\ic z}\UR(z) \partial_z \phi^{\s}(z).
\end{equation}

We introduce $\varphi_0(z) = \ee^{2\pi \ic z}\UR(z)  \partial_z
\phi^{\s}(z)$. We first claim that $\varphi_0 \in \EInt_{3-\eH}$.
Indeed, we note that $\partial_z \phi^{\s} \in \EInt_{\eH-1}$ and we
recall that
\begin{equation}
\UR(z) = -\int_{0}^{1} (1-\lambda) D^2 \gg(\phi^{\s} + \lambda(\phi^{\u}- \phi^{\s}))\,d\lambda \big (\phi^{\u}-\phi^{\s}\big ).
\end{equation}
The claim follows from the facts that, by \eqref{bounddifexp1}, $|z^{-\eH} \ee^{2\pi \ic z}(\phi^{\u} - \phi^{\s})|$ is bounded and moreover
$|D^{2}\gg(\phi^{\s} + \lambda(\phi^{\u}- \phi^{\s}))| \leq K |z|^{\eH-4}$ if $z\in E_{\g,\r}$ (which can be proved as in~\eqref{D2gR}).

%Henceforth, since by \eqref{bounddifexp1}, $|z^{-r-2} \ee^{2\pi \ic z}(\phi^{\u} - \phi^{\s})|$ is bounded and moreover
%$|D^{2}\gg(\phi^{\s} + \lambda(\phi^{\u}- \phi^{\s}))| \leq K |z|^{-r(n-2)}$ if $z\in E_{\g,\r}$ the claim is proved since $r(n-1)=2$.

It is clear that, a particular solution $\spli$ of \eqref{eqeta} is
given by a solution of
\begin{equation*}
\spli = \inlopin(\varphi_0) + \mathcal{G}(\spli)
\end{equation*}
being $\mathcal{G}(\spli) = \inlopin (\IL \cdot \spli)$.

First we observe that, by Lemma ~\ref{invlopinlemma}, the
independent term  $\inlopin(\varphi_0)\in \EInt_{1-\eH}$. Secondly
we check that $(\text{Id} - \mathcal{G})$ is invertible in
$\EInt_{1-\eH}$. Let $\psi \in \EInt_{1-\eH}$.  Since $\IL \in
\EInt_{\frac{r}{2}+2}$ for any $n\geq 2$, we have that $\IL \cdot
\psi \in \EInt_{3+\frac{r}{2}-\eH}$ and consequently, by
Lemma~\ref{invlopinlemma}, $\mathcal{G}(\psi)\in
\EInt_{1+\frac{r}{2} -\eH}$ and moreover
\begin{equation*}
\Vert \mathcal{G}(\psi)\Vert_{1-\eH} \leq \r^{-r/2} \Vert \mathcal{G}(\psi)\Vert_{1+\frac{r}{2}-\eH} \leq \r_0^{-r/2} C \Vert \psi \Vert_{1-\eH} \leq \frac{1}{2} \Vert \psi \Vert_{1-\eH}.
\end{equation*}
This implies that the norm of the linear operator $\mathcal{G}: \EInt_{1-\eH} \to \EInt_{1-\eH}$ is less than one and therefore
$\text{Id} - \mathcal{G}$ is invertible.
To this end, we can write $\spli$ as
\begin{equation*}
\spli = \big(\text{Id}-\mathcal{G}\big)^{-1}(\inlopin(\varphi_0))
\end{equation*}
and we deduce that $\spli \in \EInt_{1-\eH}$ and $\Vert \spli \Vert_{1-\eH} \leq 2 \Vert \inlopin(\varphi_0) \Vert_{1-\eH}$ which implies the result for $\varpine$.

The existence and properties of $\varsine$ follow from the ones for
$\varsin$ in Lemma~\ref{sollinearequationinner}.
\end{proof}

\section*{Acknowledgments}
P. Mart\'{\i}n was supported in part by the Generalitat de Catalunya grant 2009SGR859 and by the Spanish
MCyT/FEDER grant MTM2009-06973. I. Baldom\'a was supported by
 the Spanish Grant MEC-FEDER MTM2006-05849/Consolider, the Spanish Grant MTM2010-16425 and the Catalan SGR grant 2009SGR859.

\bibliography{biblio_inner}
\end{document}